\documentclass[12pt]{amsart}
\usepackage{pifont}
\usepackage{amsfonts}
\usepackage{amsmath}
\usepackage{amssymb}
\usepackage{amsmath,bm}
\usepackage{cite}
\usepackage{color}
\newtheorem{theorem}{Theorem}[section]

\newtheorem{lemma}[theorem]{Lemma}

\newtheorem{example}{Example}[section]
\theoremstyle{definition}

\begin{document}

\title{Spreading speeds of KPP-type nonlocal dispersal in heterogeneous media}

\author{Xing Liang}
\author {Tao Zhou}
\address{School of Mathematical Sciences and Wu Wen-Tsun Key Laboratory of Mathematics, University of Science and Technology of China, Hefei, Anhui 230026, China}
\date{17th April, 2018 }
\maketitle

\begin{abstract}
 \end{abstract}

\keywords{Nonlocal dispersal, generalized principal eigenvalue, spreading speed, heterogeneous media}

\section{Introduction}

In this paper, we focus on the large time behavior of the solution of the following problem:
\begin{equation}\label{1.1}
  \left\{
   \begin{aligned}
   u_{t}(t,x)=\int_{\mathbf{R}}K(x,y)u(t,y)dy-b(x)u(t,x)+f(x,u)\ \ t>0,x\in\mathbf{R},  \\
  0\leq u(0,x)=u_{0}(x)\leq1,\ u_{0}\neq0\ \text{with compact support,} \\
   \end{aligned}
   \right.
  \end{equation}
where function $K(x,y)$ represents the dispersal kernel, and $b(x)=\int_{R}K(x,y)dy$.
We assume that $f$ satisfies some KPP-type conditions. This will be told in detail later. A simple
example is $f(x,s)=s(1-s)$. Another type of dispersal is so-called random dispersal in the following form:
\begin{equation}
 \label{continuousequation} \left\{
   \begin{aligned}
   \partial_{t}u=d(x)\partial_{xx}u+q(x)\partial_{x}u+f(x,u)\ \ t>0,x\in\mathbf{R}, \\
   0\leq u(0,x)\leq1,\{x:u(0,x)\neq0\}\neq\emptyset\ \ \text{is bounded} .\\
   \end{aligned}
   \right.
  \end{equation}
The pioneer works on the dynamics of the type of equations like
\eqref{continuousequation} were done by Fisher \cite{F1} and Kolmogorov, Petrovsky, Piskunov \cite{K1} in the homogeneous case:
$$\partial_{t}u=\partial_{xx}u+f(u),$$
where $f\in\mathcal{C}^{1}[0,1]$, $f(0)=f(1)=0$. In fact, in \cite{F1,K1}, they proved the existence of the minimal wave speed in the case where $f(s)>0$
and $f^{\prime}(s)\leq f^{\prime}(0)s$ for any $s\in(0,1)$.
 Moreover, in the homogeneous case,
Aronson and Weinberger \cite{A1} proved that if
$f^{\prime}(0)>0$ and $f(s)>0$ for any $s\in(0,1)$, then there exists
$\omega^{\ast}>0$ such that
\begin{equation*}
  \left\{
   \begin{aligned}
   &\text{for all}\ \ \ \ \omega>\omega^{\ast},&\displaystyle{\lim_{t\rightarrow\infty}}\sup \limits_{x\geq\omega t}|u(t,x)|=0,\\
   &\text{for all}\ \ \omega\in(0,\omega^{\ast}),\ \ &\displaystyle{\lim_{t\rightarrow\infty}}\sup \limits_{0\leq x\leq\omega t}|u(t,x)-1|=0.\\
   \end{aligned}
   \right.
  \end{equation*}
A similar result still holds if $x\leq0$. An easy corollary is
$\displaystyle{\lim_{t\rightarrow\infty}}u(t,x+\omega t)=0$ if $\omega>\omega^{\ast}$ and
$\displaystyle{\lim_{t\rightarrow\infty}}u(t,x+\omega t)=1$ locally uniform in $x\in\mathbf{R}$ if $0\leq\omega<\omega^{\ast}$.
This result is called spreading property and $\omega^{\ast}$ is called
speading speed.

In the past decades, the spreading properties in heterogeneous media got increasing attentions of mathematicians.
The propagation problems in (spatially) periodic media, one simplest heterogenous case,
were considered by mathematicians widely. Applying the approach of
probability,  \cite{GF} first proved the existence of spreading speeds for one-dimensional
KPP-type reaction-diffusion equations in periodic media. \cite{SKT, X} gave the definition of the spatially
periodic traveling waves independently, and then \cite{HZ} proved the existence of
the spatially periodic traveling waves of KPP-type equations in the distributional
sense. In a series of works (e.g.\cite{BH, BHR, BHN}), Berestycki, Hamel and their colleagues investigated
the traveling waves and spreading speeds of KPP-type reaction-diffusion equations
in high-dimensional periodic media.

Besides above works, more general
frameworks are provided by \cite{LZ,W} to study spreading properties for more general diffusion systems in periodic media.

However, there are only a few works on the spreading properties
of KPP-type equations in more complicated media.
Berestycki, Hamel and Nadirashvili \cite{BHN2}  investigated spreading properties in higher dimension for the homogeneous equation in general unbounded domains.
Particularly, in \cite{BHN2}, the concepts of lower and upper spreading speeds were introduced. Then Berestycki and Nadin \cite{B2} also
introduced these two speeds again for \eqref{continuousequation} to study the spreading property. Precisely, for one-dimensional equation \eqref{continuousequation},
the upper and lower spreading speeds are defined by$$\omega^{\ast}:=\inf\{\omega\geq0,\ \displaystyle{\lim_{t\rightarrow\infty}}\sup \limits_{x\geq\omega t}|u(t,x)|=0\},$$
$$\omega_{\ast}:=\sup\{\omega\geq0,\ \displaystyle{\lim_{t\rightarrow\infty}}\sup \limits_{0\leq x\leq\omega t}|u(t,x)-1|=0\}.$$
They gave a sharp estimate on $\omega_{\ast},\ \omega^{\ast}$ by constructing
$\underline{\omega},\ \overline{\omega}$ where $\underline{\omega},\ \overline{\omega}$
are represented by two generalized principal eigenvalues (see Definition \ref{def2.1}) of the linearized equation such that
$$\underline{\omega}\leq\omega_{\ast}\leq\omega^{\ast}\leq\overline{\omega}.$$
Furthermore, they showed that if the coefficients are (asymptotically) almost periodic or random stationary ergodic, then $\underline{\omega}=\overline{\omega}$,
and hence $\omega_{\ast}=\omega^{\ast}$ is exactly the spreading speed. Most  recently, they also investigated multidimensional
and space-time heterogeneous case in \cite{B3}.
In fact, Shen (see e.g.\cite{s1,s2,s3}) also introduced the concepts of
lower and upper spreading speeds to study the spreading speeds of KPP-type equations in space-time heterogeneous media.
In \cite{LZh}, the authors obtained similar conclusions for spatial discrete equation.
Moreover, they proved that the spreading speeds in the positive and negative directions are identical even if $f(x,u)$ is not invariant with respect to the reflection.

In this paper, we investigate the spreading properties for \eqref{1.1} in general heterogeneous media.
Motivated by \cite{B2}, we establish the theory of generalized principal eigenvalues of linear
nonlocal operator to estimate the lower and upper spreading speeds $\omega_{\ast}, \omega^{\ast}$.
Aiming to estimate the spreading speeds through the principal eigenvalues, we also develop some homogenization techniques for nonlocal dispersal equations.
Then we  prove that $\omega_{\ast}=\omega^{\ast}$  in the case where the media is almost periodic.
Finally, in the case where $f_{s}^{\prime}(\cdot,0)$ is almost periodic and $K(x,y)=K(y,x)$,
we show that the spreading speeds in the positive and negative directions are identical even if  $ f_{s}^{\prime}(x,0)$ is not invariant with respect to the reflection.


\section{Preliminary: Definitions, notions, results}
In this paper, we always assume that $u_{0}\in C(\mathbf{R})$ with compact support and $u_{0}(0)>0$,
and that the reaction term $f$
satisfies $f(x,0)\equiv f(x,1)\equiv0$,
$0<\inf\limits_{x\in\mathbf{R}}f(x,s)\leq f(x,s)\leq f_{s}^{\prime}(x,0)s$
for any $s\in (0,1)$, $f_{s}^{\prime}(\cdot,0)\in C(\mathbf{R})$ and
$f(x,\cdot)\in\mathcal{C}^{1+\gamma}([0,1])$ uniformly with respect to $x\in\mathbf{R}$,
that is, $\sup\limits_{x\in\mathbf{R}}\|f(x,\cdot)\|_{\mathcal{C}^{1+\gamma}}<+\infty$.
Let $K\phi(x):=\int_{\mathbf{R}}K(x,y)\phi(y)dy$ with $K(x,y)\geq0\ \forall(x,y)\in\mathbf{R}^{2}$.
We list the following assumptions for the kernel $K$:\\
(K1) $K:C(\mathbf{R})\to C(\mathbf{R})$, and $\int_{\mathbf{R}}K(\cdot,\cdot-\xi)e^{-p\xi}d\xi\in L^{\infty}(\mathbf{R})\cap C(\mathbf{R})$ for any $p\in\mathbf{R}$.\\
(K2) For any fixed $p\in\mathbf{R}$, $K(x,x-\cdot)e^{-p\cdot}$ is uniformly integrable w.r.t. $x$,
i.e.,
$\forall \sigma>0,\ \exists R>0$ s.t.
$$\sup\limits_{x\in\mathbf{R}}\int_{B_{R}^{c}(0)}K(x,x-\xi)e^{-p\xi}d\xi<\sigma.$$\\
(K3) There exist $\delta_{0}>0$, $\eta_{0}>0$ and
positive constant $C$ depending on $\delta_{0}, \eta_{0}$ s.t.
$K[\phi\chi_{\eta+\delta_{0}}](x)>C\min\limits_{|y|\leq\eta}\phi(y)\ \forall \eta\geq\eta_{0},\
|x|\leq\eta+\delta_{0}$ and $\phi\in C(\mathbf{R})$ with $\phi\geq0$, where
\begin{equation*}
\chi_{r}(x)=\left\{
   \begin{aligned}
   1,\ |x|\leq r,\\
   0,\ |x|>r.\\
   \end{aligned}
   \right.
\end{equation*}\\
(K4) $k(p):=\liminf \limits_{R\to \infty}\inf\limits_{x\in B_{R}^{c}}\{\int_{\mathbf{R}}K(x,x-\xi)e^{-p\xi}d\xi-b(x)+f_{s}^{\prime}(x,0)\}>0\ \forall p\in\mathbf{R},$
and
\begin{equation*}
\displaystyle{\lim_{p\rightarrow+\infty}}
\frac{k(\pm p)}{p}=
\displaystyle{\lim_{p\rightarrow 0^{+}}}
\frac{k(\pm p)}{p}=+\infty,
\end{equation*}
where $b(x)=\int_{R}K(x,y)dy$.\\
(K5) $e^{Kt}u_{0}(x)=\sum\limits_{n=0}^{\infty}\frac{t^{n}K^{n}u_{0}}{n!}(x)>0\ \forall x\in\mathbf{R},\ t>0$.

For any $z\in\mathbf{R}$. Consider space
$$X_{z}=\{\phi\in C(\mathbf{R})|\sup\limits_{x\in\mathbf{R}}e^{-|x-z|}|\phi(x)|<+\infty\}$$
equipped with the norm $\|\phi\|_{z}=\sup\limits_{x\in\mathbf{R}}e^{-|x-z|}|\phi(x)|$. It is easy to
verify that $\int_{\mathbf{R}}K(\cdot,y)\phi(y)dy\in X_{z}$ if $\phi\in X_{z}$ by (K1).
By general nonlinear semigroup theory (see \cite{He} or \cite{P}), \eqref{1.1} has a unique (local) solution $u(t,x)$ with initial value $u(0,x)\in X_{z}$.
Define the linear bounded operator $\mathcal{L}:X_{z}\to X_{z}$ by $\mathcal{L}\phi(x)=K\phi(x)-a(x)\phi(x)$,
where $a(\cdot)=b(\cdot)+f_{s}^{\prime}(\cdot,0)\in L^{\infty}(\mathbf{R})\cap C(\mathbf{R})$.
Moreover, let
$L_{p}\phi(x)=e^{-px}\mathcal{L}(e^{p\cdot}\phi)(x)=e^{-px}\int_{\mathbf{R}}K(x,y)e^{py}\phi(y)dy-a(x)\phi(x).$

\newtheorem{defn}{Definition}[section]
\begin{defn}\label{def2.1}
The generalized principal eigenvalues associated with operator $L_{p}$ on $I_{R}:=(R,+\infty)$, where $R\in\{-\infty\}\cup\mathbf{R}$, are:
$$\underline{{\lambda}_{1}}(p,R):=\sup\{\lambda|\ \exists\ \phi\in{\mathcal{A}},\ \text{s.t.}\
L_{p}\phi(x)\geq \lambda\phi(x)\ \forall\ x\in I_{R}\},$$
$$\overline{{\lambda}_{1}}(p,R):=\inf\{\lambda|\ \exists\ \phi\in{\mathcal{A}},\ \text{s.t.}\
L_{p}\phi(x)\leq \lambda\phi(x)\ \forall\ x\in I_{R}\},$$
where ${\mathcal{A}}$ is a set of admissible test functions:
$${\mathcal{A}}:=\{\phi\in C(R)|0<\inf \limits_{x\in\mathbf{R}}\phi(x)\leq\sup \limits_{x\in\mathbf{R}}\phi(x)<+\infty,\ \phi\ \text{is uniformly continuous}\}.$$
\end{defn}
We use the convention that $\underline{{\lambda}_{1}}(p,R)=-\infty$ if
$\{\lambda|\ \exists\ \phi\in{\mathcal{A}},\ \text{s.t.}\
L_{p}\phi(x)\geq \lambda\phi(x)\ \forall\ x\in I_{R}\}$ is empty and
$\overline{{\lambda}_{1}}(p,R)=+\infty$
if $\{\lambda|\ \exists\ \phi\in{\mathcal{A}},\ \text{s.t.}\
L_{p}\phi(x)\leq \lambda\phi(x)\ \forall\ x\in I_{R}\}$
is empty.
An important relation between $\underline{{\lambda}_{1}}(p,R)$ and $\overline{{\lambda}_{1}}(p,R)$
is that

\newtheorem{prop}{Proposition}[section]
\begin{prop}\label{prop2.1}
Assume that $\int_{\mathbf{R}}K(x,x-\xi)e^{-p\xi}d\xi<+\infty$.
For all $R\in\{-\infty\}\cup\mathbf{R}$, we have
$$\underline{{\lambda}_{1}}(p,R)\leq\overline{{\lambda}_{1}}(p,R)\ \forall p\in\mathbf{R}.$$
\end{prop}
This proposition and Definition \ref{def2.1} yield the following corollary immediately.

\newtheorem{cor}{Corollary}[section]
\begin{cor}\label{cor2.1}
Let the assumptions in Proposition \ref{prop2.1} hold. If there exist $p,\lambda\in\mathbf{R}$, and $\phi\in\mathcal{A}$
such that $L_{p}\phi(x)=\lambda\phi(x)$ $\forall x\in I_{R}$, then
$$\lambda=\underline{{\lambda}_{1}}(p,R)=\overline{{\lambda}_{1}}(p,R).$$
\end{cor}
We write $\underline{{\lambda}_{1}}(p,R,a(\cdot))$ and $\overline{{\lambda}_{1}}(p,R,a(\cdot))$ to
emphasize that the generalized principal eigenvalues are related to $a(x)$. Then one can use the Definition \ref{def2.1} to verify
\begin{prop}\label{prop2.2}
For any $R\in\{-\infty\}\cup\mathbf{R}$, $p\in\mathbf{R}$, we have
$$|\underline{{\lambda}_{1}}(p,R,a(\cdot))-\underline{{\lambda}_{1}}(p,R,a^{\prime}(\cdot))|
\leq\sup\limits_{x\in I_{R}}\|a^{\prime}-a\|_{{\infty}},$$
$$|\overline{{\lambda}_{1}}(p,R,a(\cdot))-\overline{{\lambda}_{1}}(p,R,a^{\prime}(\cdot))|
\leq\sup\limits_{x\in I_{R}}\|a^{\prime}-a\|_{{\infty}}.$$
\end{prop}
It is easy to see that $\underline{{\lambda}_{1}}(p,R)$ is increasing in $R$, and
$\overline{{\lambda}_{1}}(p,R)$ is decreasing in $R$. By Proposition \ref{prop2.1}, one can define:
\begin{equation}\label{2.2}
\overline{H}(p):=\displaystyle{\lim_{R\rightarrow +\infty}\overline{{\lambda}_{1}}(p,R)},\ \text{and}\ \ \underline{H}(p):=\displaystyle{\lim_{R\rightarrow +\infty}\underline{{\lambda}_{1}}(p,R)},\ \forall \ p\in\mathbf{R}.
\end{equation}

\begin{prop}\label{prop2.3}
Assume that (K1) and (K4) hold. The functions $0<\underline{H}(p)\leq\overline{H}(p)$
are locally Lipschitz continuous, and
$$\displaystyle{\lim_{p\rightarrow{0}^{+}}}\frac{\underline{H}(\pm p)}{p}=
\displaystyle{\lim_{p\rightarrow+\infty}}\frac{\underline{H}(\pm p)}{p}=+\infty.$$
\end{prop}

Now, as in \cite{B2} we can define the speeds $\underline{\omega}$ and $\overline{\omega}$:
\begin{equation}\label{2.3}
\underline{\omega}:=\min \limits_{p>0}\frac{\underline{H}(-p)}{p},\  \text{and}\
\overline{\omega}:=\min \limits_{p>0}\frac{\overline{H}(-p)}{p}.
\end{equation}

\newtheorem{rem}{Remark}[section]
\begin{rem}\label{re2.1}
If we set $K^{-}(x,y)=K(-x,-y), f^{-}(x,u)=f(-x,u)$, then, as we did before, one can still define
$\underline{{\lambda}^{-}_{1}}(p,R),\overline{{\lambda}^{-}_{1}}(p,R),\
\underline{H}^{-}(p),\overline{H}^{-}(p),\underline{\omega}^{-} \text{and}\ \overline{\omega}^{-}$
associated with $\mathcal{L}^{-}$, where $\mathcal{L}^{-}:X_{z}\to X_{z}$ is defined by
$\mathcal{L}^{-}\phi(x)=K^{-}\phi(x)-a^{-}(x)\phi(x)$ with
$K^{-}\phi(x)=\int_{\mathbf{R}}K^{-}(x,y)\phi(y)dy, a^{-}(x)=b(-x)+f^{\prime}(-x,0)$.
Moreover, we have
$\underline{{\lambda}^{-}_{1}}(p,R)\leq\overline{{\lambda}^{-}_{1}}(p,R),\
0<\underline{H}^{-}(p)\leq\overline{H}^{-}(p)$ and
$0<\underline{\omega}^{-}\leq\overline{\omega}^{-}$.
\end{rem}

The main result of this paper is as following:
\newtheorem{thm}{Theorem}[section]
\begin{thm}\label{thm2.1}
Let $u(t,x)$ be a solution of \eqref{1.1}. Then:\\
1)\ For all $\omega>\overline{\omega}$, $\displaystyle{\lim_{t\rightarrow+\infty}}\sup \limits_{x\geq\omega t}|u(t,x)|=0$ provided that (K1) and (K4) hold;\\
2)\ For all $0\leq\omega<\underline{\omega}$, $\displaystyle{\lim_{t\rightarrow+\infty}}\sup \limits_{0\leq x\leq\omega t}|u(t,x)-1|=0$ provided that (K1)-(K5) hold.
\end{thm}

\section{Properties of generalized principal eigenvalues}

\begin{proof}[Proof of Proposition \ref{prop2.1}]
Assume by contradiction that
$\underline{{\lambda}_{1}}(p,R)>\overline{{\lambda}_{1}}(p,R)$. Then there exist $\lambda\in\mathbf{R}$ and
$\varepsilon>0$ such that
$\underline{{\lambda}_{1}}(p,R)>\lambda>\lambda-2\varepsilon>\overline{{\lambda}_{1}}(p,R)$,
and $\phi,\psi\in{\mathcal{A}}$ such that
\begin{equation*}
\left\{
   \begin{aligned}
   \int_{\mathbf{R}}K(x,y)e^{p(y-x)}\psi(y)dy-a(x)\psi(x)\geq\lambda\psi(x),\ x\in I_{R},\\
   \int_{\mathbf{R}}K(x,y)e^{p(y-x)}\phi(y)dy-a(x)\phi(x)\leq(\lambda-2\varepsilon)\phi(x),\ x\in I_{R}.\\
   \end{aligned}
   \right.
\end{equation*}
Therefore,
\begin{equation}\label{3.1}
\left\{
   \begin{aligned}
   \int_{\mathbf{R}}K(x+R_{n},y+R_{n})e^{p(y-x)}\psi_{n}(y)dy-a_{n}(x)\psi_{n}(x)\geq\lambda\psi_{n}(x),\ x\in I_{R-R_{n}},\\
   \int_{\mathbf{R}}K(x+R_{n},y+R_{n})e^{p(y-x)}\phi_{n}(y)dy-a_{n}(x)\phi_{n}(x)\leq(\lambda-2\varepsilon)\phi_{n}(x),\ x\in I_{R-R_{n}},\\
   \end{aligned}
   \right.
\end{equation}
where $\psi_{n}(x)=\psi(x+R_{n}),\phi_{n}(x)=\phi(x+R_{n}), a_{n}(x)=a(x+R_{n})$, and $R_{n}\to+\infty$ as $n\to\infty$.
Then by Arzel¨¤-Ascoli Theorem, there exists a subsequence (still denoted by $R_{n}$) and
bounded continuous  functions $\psi_{\infty}, \phi_{\infty}$ such that
$\psi_{n}\to\psi_{\infty}, \phi_{n}\to\phi_{\infty}$ locally uniform.
For any fixed $x\in\mathbf{R}$, there exists a subsequence $\{R_{n}(x)\}\subset\{R_{n}\}$ such that
$a(x+R_{n}(x))\to a_{\infty}(x)$.
Hence
\begin{equation}\label{3.1a}
\liminf\limits_{n\to+\infty}\int_{\mathbf{R}}K(x+R_{n}(x),y+R_{n}(x))e^{p(y-x)}\psi_{n}(y)dy
-a_{\infty}(x)\psi_{\infty}(x)\geq\lambda\psi_{\infty}(x),\ x\in\mathbf{R},
\end{equation}
\begin{equation}\label{3.1b}
\limsup\limits_{n\to+\infty}\int_{\mathbf{R}}K(x+R_{n}(x),y+R_{n}(x))e^{p(y-x)}\phi_{n}(y)dy
-a_{\infty}(x)\phi_{\infty}(x)\leq(\lambda-2\varepsilon)\phi_{\infty}(x),\ x\in\mathbf{R}.
\end{equation}
Let
$\gamma=\inf \limits_{x\in\mathbf{R}}\frac{\phi_{\infty}(x)}{\psi_{\infty}(x)}\geq\frac{\inf \limits_{x\in\mathbf{R}}\phi(x)}{\sup \limits_{x\in\mathbf{R}}\psi(x)}>0$,
$z_{n}=\phi_{n}-\gamma\psi_{n}$, and
$z=\phi_{\infty}-\gamma\psi_{\infty}\geq0$. Then $z_{n}\to z_{\infty}$ locally uniform, and
$\inf \limits_{x\in\mathbf{R}}z_{\infty}(x)=0$. \eqref{3.1b}-\eqref{3.1a} yields that
\begin{equation}\label{3.2}
\begin{split}
\limsup\limits_{n\to+\infty}\int_{\mathbf{R}}K(x+R_{n}(x),y+R_{n}(x))e^{p(y-x)}z_{n}(y)dy
&\leq\limsup\limits_{n\to+\infty}\int_{\mathbf{R}}K(x+R_{n}(x),y+R_{n}(x))e^{p(y-x)}\phi_{n}(y)dy\\
&\quad-\liminf\limits_{n\to+\infty}\int_{\mathbf{R}}K(x+R_{n}(x),y+R_{n}(x))e^{p(y-x)}\gamma\psi_{n}(y)dy\\
&\leq(\lambda+a_{\infty}(x))z(x)-2\varepsilon\phi_{\infty}(x)\\
&\leq(\lambda+a_{\infty}(x))z(x)-2\varepsilon\inf \limits_{x\in\mathbf{R}}\phi(x).
\end{split}
\end{equation}
Choose $x_{0}$ such that
$(\lambda+a_{\infty}(x_{0}))z(x_{0})\leq\varepsilon\inf \limits_{x\in\mathbf{R}}\phi(x)$ since
$a_{\infty}\in L^{\infty}$ and $\inf \limits_{x\in\mathbf{R}}z_{\infty}(x)=0$. Then, at $x_{0}$,
the right hand side of \eqref{3.2} $\leq-\varepsilon\inf \limits_{x\in\mathbf{R}}\phi(x)$.
On the other hand, there exists $R>0$ depending on $x_{0}$ such that
$\int_{B^{c}_{R}(x_{0})}K(x_{0}+R_{n}(x_{0}),y+R_{n}(x_{0}))e^{p(y-x_{0})}z_{n}(y)dy
\leq\frac{1}{2}\varepsilon\inf \limits_{x\in\mathbf{R}}\phi(x)$ since $z_{n}$ is bounded.
Hence at $x_{0}$, we have
\begin{equation*}
\begin{split}
\text{LHS of}\ \eqref{3.2}
&\geq\limsup\limits_{n\to+\infty}\int_{B_{R}(x_{0})}K(x_{0}+R_{n}(x_{0}),y+R_{n}(x_{0}))
e^{p(y-x_{0})}z_{n}(y)dy-\frac{1}{2}\varepsilon\inf \limits_{x\in\mathbf{R}}\phi(x)\\
&\geq\limsup\limits_{n\to+\infty}\int_{B_{R}(x_{0})}-K(x_{0}+R_{n}(x_{0}),y+R_{n}(x_{0}))
e^{p(y-x_{0})}dy\sup\limits_{y\in B_{R}(x_{0})}|z_{n}(y)|-\frac{1}{2}\varepsilon\inf \limits_{x\in\mathbf{R}}\phi(x)\\
&\geq-\frac{1}{2}\varepsilon\inf \limits_{x\in\mathbf{R}}\phi(x).
\end{split}
\end{equation*}
which contradicts RHS of \eqref{3.2} $\leq-\varepsilon\inf \limits_{x\in\mathbf{R}}\phi(x)$!
\end{proof}

Next we will prove that $\overline{H}(p)$ and $\underline{H}(p)$ are locally Lipschitz continuous by showing
that $\underline{{\lambda}_{1}}(p,R)$ and $\overline{{\lambda}_{1}}(p,R)$ are locally Lipschitz continuous with
respect to $p$ uniformly in $R\in\mathbf{R}$.

\begin{lemma}\label{lem3.1}
For any $\alpha,\beta\in(0,1), \alpha+\beta=1$, $p,q\in\mathbf{R}$, and $\phi(x)>0, \psi(x)>0$ for any
$x\in\mathbf{R}$ with $\lim\limits_{|x|\to+\infty}\frac{\ln\phi(x)}{|x|}<+\infty$,
$\lim\limits_{|x|\to+\infty}\frac{\ln\psi(x)}{|x|}<+\infty$, we have
 \begin{equation}\label{l3.1.1}
 \frac{(L_{\alpha p+\beta q}{\phi}^{\alpha}{\psi}^{\beta})}{{\phi}^{\alpha}{\psi}^{\beta}}(x)
\leq\alpha\frac{(L_{p}\phi)}{{\phi}}(x)+\beta\frac{(L_{q}\psi)}{{\psi}}(x),\ x\in\mathbf{R}.
 \end{equation}
\end{lemma}

\begin{proof}
\begin{equation*}
\begin{split}
\frac{\int_{\mathbf{R}}K(x,y)e^{(\alpha p+\beta q)(y-x)}\phi^{\alpha}(y)\psi^{\beta}(y)dy}
{\phi^{\alpha}(x)\psi^{\beta}(x)}
&=\int_{\mathbf{R}}(K(x,y)e^{p(y-x)}\frac{\phi(y)}{\phi(x)})^{\alpha}
(K(x,y)e^{q(y-x)}\frac{\psi(y)}{\psi(x)})^{\beta}dy\\
&\leq\big(\int_{\mathbf{R}}K(x,y)e^{p(y-x)}\frac{\phi(y)}{\phi(x)}dy\big)^{\alpha}
\big(\int_{\mathbf{R}}K(x,y)e^{p(y-x)}\frac{\psi(y)}{\psi(x)}dy\big)^{\beta}\\
&\leq\alpha\int_{\mathbf{R}}K(x,y)e^{p(y-x)}\frac{\phi(y)}{\phi(x)}dy+
\beta\int_{\mathbf{R}}K(x,y)e^{p(y-x)}\frac{\psi(y)}{\psi(x)}dy.
\end{split}
\end{equation*}
Hence
$$\frac{(L_{\alpha p+\beta q}{\phi}^{\alpha}{\psi}^{\beta})}{{\phi}^{\alpha}{\psi}^{\beta}}
\leq\alpha\frac{(L_{p}\phi)}{{\phi}}+\beta\frac{(L_{q}\psi)}{{\psi}}.$$
\end{proof}

One can easily obtain the following corollary by the definition of $\overline{{\lambda}_{1}}(p,R)$ immediately.
\begin{cor}\label{cor3.1}
For any $R\in\{-\infty\}\cup\mathbf{R}$. $\overline{{\lambda}_{1}}(p,R)$ is convex with respect to $p$. i.e.,
$$\alpha\overline{{\lambda}_{1}}({p}_{1},R)+\beta\overline{{\lambda}_{1}}({p}_{2},R)
\geq\overline{{\lambda}_{1}}(\alpha{p}_{1}+\beta{p}_{2},R).$$
\end{cor}

\begin{lemma}
Assume that (K1) holds. $\underline{{\lambda}_{1}}(p,R)$ and
$\overline{{\lambda}_{1}}(p,R)$ are locally Lipschitz continuous in $p$ and
the Lipschitz constant is independent of $R\in\{-\infty\}\cup\mathbf{R}$.
\end{lemma}
\begin{proof}
Assume that
${p}_{1}\neq{p}_{2},\ \alpha=\frac{|{p}_{2}-{p}_{1}|}{1+|{p}_{2}-{p}_{1}|},
\ \beta=\frac{1}{1+|{p}_{2}-{p}_{1}|}$. We may, without loss of generality,
assume that $|p_{1}-p_{2}|\leq\frac{1}{2}$ and that
$\underline{{\lambda}_{1}}({p}_{j},R)(j=1,2)$ and $\overline{{\lambda}_{1}}({p}_{j},R)$
are positive by adding a sufficiently large constant $M$ to ${L}_{p}$.
For any $\varepsilon>0$, there exists
$\phi_{2}\in{\mathcal{A}}$ such that
$$L_{{p}_{2}}\phi_{2}\leq(\overline{{\lambda}_{1}}({p}_{2},R)+\varepsilon)\phi_{2}\ \text{on}\ I_{R}.$$
Set $p=p_{1}+\frac{p_{1}-p_{2}}{|p_{1}-p_{2}|}$. Then $\alpha p+\beta p_{2}=p_{1}$.
From Lemma \ref{lem3.1}, we have
\begin{equation}\label{l3.2.1}
\begin{split}
\frac{(L_{p_{1}}\phi_{2}^{\beta})}{{\phi_{2}^{\beta}}}(x)
&\leq\alpha\frac{(L_{p}1)}{1}(x)+\beta\frac{(L_{p_{2}}\phi_{2})}{{\phi_{2}}}(x)\\
&\leq\alpha\big(\int_{\mathbf{R}}K(x,y)e^{p(y-x)}dy-a(x)\big)+
 \beta(\overline{{\lambda}_{1}}({p}_{2},R)+\varepsilon)\\
&\leq\alpha\big(\int_{\mathbf{R}}K(x,y)e^{p(y-x)}dy-a(x)\big)+(\overline{{\lambda}_{1}}({p}_{2},R)+\varepsilon)\\
&\leq\overline{{\lambda}_{1}}({p}_{2},R)+\varepsilon+
\frac{|{p}_{2}-{p}_{1}|}{1+|{p}_{2}-{p}_{1}|}\big(\int_{\mathbf{R}}K(x,y)e^{p(y-x)}dy+f_{s}^{\prime}(x,0)\big)
\end{split}
\end{equation}
for $x\in I_{R}$. Therefore,
$$\frac{(L_{p_{1}}\phi_{2}^{\beta})}{{\phi_{2}^{\beta}}}(x)
 \leq\overline{{\lambda}_{1}}({p}_{2},R)+\varepsilon+C_{0}|p_{1}-p_{2}|,\ \forall x\in I_{R},$$
where $C_{0}$ depends on $p_{1}, K$, and $f_{s}^{\prime}(x,0)$ but independent of $R$.
Take $\varepsilon\to0$ and note that $\phi_{2}^{\beta}\in\mathcal{A}$ since $\phi_{2}\in\mathcal{A}$.
The definition of $\overline{{\lambda}_{1}}({p}_{j},R)$ yields that
$\overline{{\lambda}_{1}}({p}_{1},R)\leq\overline{{\lambda}_{1}}({p}_{2},R)+C_{0}|p_{1}-p_{2}|.$
By the symmetry, we have
$|\overline{{\lambda}_{1}}({p}_{1},R)-\overline{{\lambda}_{1}}({p}_{2},R)|\leq C_{0}|p_{1}-p_{2}|$.

Similarly, there exists $\psi\in\mathcal{A}$ such that
$$L_{{p}_{1}}\psi(x)\geq(\underline{{\lambda}_{1}}({p}_{1},R)-\varepsilon)\psi(x)\ \text{on}\ I_{R}.$$
Set $\phi_{1}\equiv1, \phi_{2}=\psi^{\frac{1}{\beta}}, p=p_{1}+sgn(p_{1}-p_{2}),$ i.e.,
$\alpha p+\beta p_{2}=p_{1}$. Then by Lemma \ref{lem3.1}
\begin{equation*}
\begin{split}
\frac{(L_{p_{2}}\psi^{\frac{1}{\beta}})}{\psi^{\frac{1}{\beta}}}(x)
&\geq\frac{1}{\beta}\frac{L_{p_{1}}\psi}{\psi}(x)-\frac{\alpha}{\beta}\frac{L_{p}1}{1}(x)\\
&\geq\underline{{\lambda}_{1}}({p}_{1},R)-\varepsilon-
{|{p}_{2}-{p}_{1}|}\big(\int_{\mathbf{R}}K(x,y)e^{p(y-x)}dy+f_{s}^{\prime}(x,0)\big).
\end{split}
\end{equation*}
Hence one can still obtain that
$|\underline{{\lambda}_{1}}({p}_{1},R)-\underline{{\lambda}_{1}}({p}_{2},R)|\leq C_{1}|p_{1}-p_{2}|$
for some constant $C_{1}$.
\end{proof}

\begin{proof}[Proof of Proposition \ref{prop2.3}]
It is easy to see that $\overline{H}(p)$ and $\underline{H}(p)$ are locally Lipschitz continuous by
Lemma \ref{lem3.1}. Take $\phi\equiv 1$ as a test function. Then the proof is closed by (K4).
\end{proof}

\section{Proof of the spreading property}

In this section, we always assume that $b(\cdot)\in L^{\infty}(\mathbf{R})\cap C(\mathbf{R})$, and
$a(x)=b(x)-f_{s}^{\prime}(x,0)$.

\subsection{Proof of The first part of Theorem \ref{thm2.1}}

We first give a useful lemma which we will need later.
\begin{lemma}\label{lem4.1}
Assume that $z$ is bounded on $I\times\mathbf{R}$ for any bounded interval $I\in[0,+\infty)$,
$z(t,x)$ is differentiable in $t\in(0,+\infty)$
and continuous on $[0,+\infty)$ for all $x\in\mathbf{R}$,
and that $z(t,x)$ satisfies
\begin{equation}\label{l4.1.1}
\left\{
   \begin{aligned}
 z_{t}(t,x)\geq\int_{\mathbf{R}}K(x,y)z(t,y)dy-a(x)z(t,x),\ & \ \ (0,+\infty)\times\mathbf{R},\\
                                z\geq0,\ & \ \ t=0.\\
   \end{aligned}
   \right.
\end{equation}
Then $z(t,x)\geq0$ for $(t,x)\in(0,+\infty)\times\mathbf{R}$.
\end{lemma}

\begin{proof}
We may, without loss of generality, assume that $0<a(x)<b(x)+F$ for some positive constant $F$.
In fact, one can consider
$\zeta(t,x)=z(t,x){e}^{ct}$ for $c$ large enough instead of $z(t,x)$.
Now we prove that $z(t,x)\geq0$. If not, then there exists
$(t_{0},x_{0})\in(0,+\infty)\times\mathbf{R}$ such that $z(t_{0},x_{0})<0$.
Let $m=\inf\limits_{[0,t_{0}]\times\mathbf{R}}z(t,x)\leq z(t_{0},x_{0})<0$. For any $\varepsilon>0$ small,
there exists $(t_{\varepsilon},x_{\varepsilon})\in(0,t_{0}]\times\mathbf{R}$ such that
$z(t_{\varepsilon},x_{\varepsilon})\leq m+\varepsilon<0$. Consider $z(t,x_{\varepsilon})$ on
$[0,t_{\varepsilon}]$. Then there must be $\tau_{\varepsilon}\in(0,t_{\varepsilon}]$ such that
$z(\tau_{\varepsilon},x_{\varepsilon})=\min\limits_{[0,t_{\varepsilon}]}z(t,x_{\varepsilon})$.
Hence $z_{t}(t,x_{\varepsilon})\leq0$ at $\tau_{\varepsilon}$. Now from \eqref{l4.1.1} we have
\begin{equation}\label{l4.1.2}
\begin{split}
a(x_{\varepsilon})z(\tau_{\varepsilon},x_{\varepsilon})
&\geq\int_{\mathbf{R}}K(x_{\varepsilon},y)z(\tau_{\varepsilon},y)dy-z_{t}(\tau_{\varepsilon},x_{\varepsilon})\\
&\geq\int_{\mathbf{R}}K(x_{\varepsilon},y)z(\tau_{\varepsilon},y)dy\\
&\geq m\int_{\mathbf{R}}K(x_{\varepsilon},y)dy.
\end{split}
\end{equation}
On the other hand,
$$a(x_{\varepsilon})z(\tau_{\varepsilon},x_{\varepsilon})\leq a(x_{\varepsilon})(m+\varepsilon)\leq(b(x_{\varepsilon})+F)(m+\varepsilon).$$
Combining this with \eqref{l4.1.2}, we have $(b(x_{\varepsilon})+F)(m+\varepsilon)\geq mb(x_{\varepsilon})$, i.e.,
$F(m+\varepsilon)+\varepsilon b(x_{\varepsilon})\geq0$, which contradicts
$m<0$ since $\varepsilon$ can be arbitrarily small. The proof is complete.
\end{proof}

\begin{rem}\label{re4.1}
Assume that $z$ is bounded on $I\times\mathbf{R}$ for any bounded interval $I\in[0,+\infty)$,
and that $z(t,x)$ is locally Lipschitz continuous on $[0,\infty)$ for any $x\in\mathbf{R}$.
We denote the left derivative by $\frac{d^{-}}{dt}$,
and assume that $\frac{d^{-}}{dt}z(t,x)$ exists for any $t>0$, $x\in\mathbf{R}$. Moreover, we assume
that $z$ satisfies
$$\left\{
   \begin{aligned}
 \frac{\partial^{-}}{\partial t}z(t,x)\geq\int_{\mathbf{R}}K(x,y)z(t,y)dy-a(t,x,z)z(t,x),\ & \ \ \ (0,+\infty)\times\mathbf{R},,\\
                                     z\geq0,\ & \ \ \ t=0,\\
   \end{aligned}
   \right.$$
where $a\in L^{\infty}((0,+\infty)\times\mathbf{R}^{2})$. Then $z(t,x)\geq0$ for $(t,x)\in(0,+\infty)\times\mathbf{R}$ still holds.
\end{rem}

\begin{cor}\label{cor4.1}
For any $\alpha\in(0,1)$. Let $v$ be the solution of
\begin{equation}\label{c4.1.1}
\left\{
   \begin{aligned}
 v_{t}(t,x)=\int_{\mathbf{R}}K(x,y)v(t,y)dy-b(x)v(t,x)+f(x,v),\ & \ \ (0,+\infty)\times\mathbf{R},\\
                                v=\alpha,\ & \ \ t=0.\\
   \end{aligned}
   \right.
\end{equation}
\end{cor}
Then $\displaystyle{\lim_{t\rightarrow +\infty}}v(t,x)=1$ uniformly in $x\in\mathbf{R}$.
\begin{proof}
First, by Lemma \ref{lem4.1}, we have $0\leq v\leq1$ on $[0,+\infty)\times\mathbf{R}$.
Let $0<c(s)=\inf\limits_{x\in\mathbf{R}}f(x,s)\leq f(x,s)\ \forall s\in(0,1)$, and $\underline{v}$ be the solution
of
$$\left\{
   \begin{aligned}
 \underline{v}^{\prime}(t)=c(\underline{v})\ & \ \ (0,+\infty),\\
                                \underline{v}=\alpha,\ & \ \ t=0.\\
   \end{aligned}
   \right.$$
Then one can easily find that $\displaystyle{\lim_{t\rightarrow +\infty}}\underline{v}(t)=1$. Let
$w=v-\underline{v}$. Then $w$ satisfies $w(0,x)=0$ and
\begin{equation*}
\begin{split}
w_{t}(t,x)
&=\int_{\mathbf{R}}K(x,y)w(t,y)dy-b(x)w(t,x)+f(x,v)-c(\underline{v})\\
&\geq\int_{\mathbf{R}}K(x,y)w(t,y)dy-b(x)w(t,x)+f(x,v)-f(x,\underline{v})\\
&\geq\int_{\mathbf{R}}K(x,y)w(t,y)dy-b(x)w(t,x)+\hat{c}(t,x)w(t,x),
\end{split}
\end{equation*}
where $\hat{c}(t,x)=\int_{0}^{1}f_{s}^{\prime}(x,\underline{v}+s(v-\underline{v}))ds$ is bounded
on $\mathbf{R}$. Hence $w(t,x)\geq0$ by Lemma \ref{lem4.1}, i.e., $v(t,x)\geq\underline{v}(t)$.
Therefore, $1\geq\displaystyle{\lim_{t\rightarrow +\infty}}v(t,x)\geq\displaystyle{\lim_{t\rightarrow +\infty}}\underline{v}(t)=1$.
\end{proof}

\begin{proof}[Proof of part 1 of Theorem \ref{thm2.1}]
For any given $\omega>\overline{\omega}$, i.e.,
$\omega>\min\limits_{p>0}\frac{\overline{H}(-p)}{p}$, there exist $p>0$ and $R$ large enough such that
$\overline{{\lambda}_{1}}(-p,R)<\omega p$. Hence for
$\delta\in(0,\omega p)$, there exists $\phi\in\mathcal{A}$ such that
$$L_{p}\phi(x)\leq(\omega p-\delta)\phi(x)\ \forall x\in I_{R},$$
i.e.,
\begin{equation}\label{t2.1a}
\mathcal{L}(e^{-p\cdot}\phi)(x)\leq(\omega p-\delta)e^{-px}\phi(x)\ \forall x\in I_{R}.
\end{equation}
Since $\inf\limits_{x\in\mathbf{R}}\phi(x)>0$ and $u_{0}$ has compact support,
we may assume that ${\phi(x)e}^{-px}\geq u_{0}(x)$ on $\mathbf{R}$ and
$\inf\limits_{x\in\mathbf{R}}\phi(x)e^{-pR}\geq1$ through multiplying by a sufficiently large constant.
Let $\psi(t,x)=\phi(x)e^{-px+(\omega p-\delta)t}(\geq\phi(x)e^{-px}$ for $(t,x)\in[0,+\infty)\times\mathbf{R}$).
Then $\psi(t,x)\geq1$ for $(t,x)\in[0,+\infty)\times(-\infty,R]$ and
$\displaystyle{\lim_{x\rightarrow \infty}}\psi(t,x)=0$ locally uniform
with respect to $t\in[0,\infty)$.
Moreover, \eqref{t2.1a} yields that
$$\psi_{t}(t,x)\geq\mathcal{L}\psi(t,x)\ \forall(t,x)\in[0,+\infty)\times I_{R}.$$
Let $v(t,x)=\min\{1,\psi(t,x)\}(\leq\psi(t,x))$ for $(t,x)\in[0,+\infty)\times\mathbf{R}$. Then
$v(t,x)=1$ on $[0,+\infty)\times(-\infty,R]$. One can verify that $v$ satisfies
$$\left\{
   \begin{aligned}
 \frac{\partial^{-}}{\partial t}v(t,x)\geq\int_{\mathbf{R}}K(x,y)v(t,y)dy-b(x)v(t,x)+f(x,v),\ & \ \ \ (0,+\infty)\times\mathbf{R},,\\
                                     v\geq0,\ & \ \ \ t=0.\\
   \end{aligned}
   \right.$$
Now let $w=v-u$. Then $w$ satisfies
$$\left\{
   \begin{aligned}
 \frac{\partial^{-}}{\partial t}w(t,x)\geq\int_{\mathbf{R}}K(x,y)w(t,y)dy-b(x)w(t,x)+\hat{c}(t,x)w(t,x),\ & \ \ \ (0,+\infty)\times\mathbf{R},\\
                                     v\geq0,\ & \ \ \ t=0.\\
   \end{aligned}
   \right.$$
By Remark \ref{re4.1}, one has $u\leq v$, in particular,
$$\sup\limits_{x\geq\omega t}u(t,x)
\leq\sup\limits_{x\geq\omega t}\phi(x)e^{-px+(\omega p-\delta)t}
\leq\sup\limits_{x\geq\omega t}\phi(x)e^{-px+(\omega p-\delta)\frac{x}{\omega}}
=\sup\limits_{x\geq\omega t}\phi(x)e^{\frac{-\delta x}{\omega}}\to0$$
as $t\to+\infty$.
\end{proof}

\begin{rem}\label{re4.2}
In fact, we need (K1) and the function $\overline{H}(p)>0$ is locally Lipschitz continuous with
$$\displaystyle{\lim_{p\rightarrow{0}^{+}}}\frac{\overline{H}(\pm p)}{p}=
\displaystyle{\lim_{p\rightarrow+\infty}}\frac{\overline{H}(\pm p)}{p}=+\infty.$$
\end{rem}

\subsection{Homogenization techniques to the equation}

In order to show the second part of Theorem \ref{thm2.1}, we will first use homogenization techniques
to consider the behavior of $v_{\varepsilon}(t,x):=u(\frac{t}{\varepsilon},\frac{x}{\varepsilon})$ as
$\varepsilon\to0$. For this reason, we need consider $\varepsilon\ln v_{\varepsilon}(t,x)$,
which is well defined since the following

\begin{lemma}\label{lem4.2}
Assume that $u$ is the solution of \eqref{1.1} and (K3) holds. Then $u(t,x)>0$ for any
$(t,x)\in(0,+\infty)\times\mathbf{R}$.
\end{lemma}
\begin{proof}
The proof follows from similar arguments to \cite[Proposition 2.2]{sz}.
\end{proof}

Denote
$z_{\varepsilon}(t,x)=\varepsilon\ln v_{\varepsilon}(t,x), K_{\varepsilon}(x,y)=K(\frac{x}{\varepsilon},\frac{y}{\varepsilon}), f_{\varepsilon}(x,s)=f(\frac{x}{\varepsilon},s),$ and
$b_{\varepsilon}(x)=b(\frac{x}{\varepsilon})=\int_{\mathbf{R}}K(\frac{x}{\varepsilon},y)dy
=\frac{1}{\varepsilon}\int_{\mathbf{R}}K_{\varepsilon}(x,y)dy$.
Then, for $(t,x)\in(0,+\infty)\times\mathbf{R}$, $v_{\varepsilon}(t,x)$ satisfies
\begin{equation}\label{vHeq}
\partial_{t} v_{\varepsilon}(t,x)=
\frac{1}{\varepsilon^{2}}\int_{\mathbf{R}}K_{\varepsilon}(x,y)v_{\varepsilon}(t,y)dy
-\frac{1}{\varepsilon}b_{\varepsilon}(x)v_{\varepsilon}(t,x)+\frac{1}{\varepsilon}f_{\varepsilon}(x,v_{\varepsilon}),
\end{equation}
and $z_{\varepsilon}$ satisfies
\begin{equation}\label{zHeq}
\partial_{t} z_{\varepsilon}(t,x)=
\frac{1}{\varepsilon}\int_{\mathbf{R}}K_{\varepsilon}(x,y)\exp\big(\frac{z_{\varepsilon}(t,y)-z_{\varepsilon}(t,x)}{\varepsilon}\big)dy
-b_{\varepsilon}(x)+\frac{f_{\varepsilon}(x,v_{\varepsilon})}{v_{\varepsilon}}.
\end{equation}

\begin{thm}\label{thm4.1}
For any compact set $Q\subseteq(0,+\infty)\times\mathbf{R}$ there exist constants $c>0$ and $\varepsilon_{0}(\leq1)$ depending on $Q$ such that $|z_{\varepsilon}(t,x)|\leq c$ for all $\varepsilon\in(0,\varepsilon_{0}),(t,x)\in Q$.
\end{thm}

To prove this theorem, we need the following
\begin{lemma}\label{lem4.3}
Let $B=\|b\|_{\infty}$ and $\sigma(t_{0})=\min\limits_{|y|\leq\eta}u(t_{0},y)$. Then we have
$$u(t,x)\geq e^{-B(t-t_{0})}\big(\sigma(t_{0})(e^{C(t-t_{0})}-1)+u(t_{0},x)\big)\ \forall\ t\geq t_{0}, |x|<\eta+{\delta_{0}}.$$
\end{lemma}
\begin{proof}
\begin{equation*}
\begin{split}
u_{t}(t,x)
&=\int_{\mathbf{R}}K(x,y)u(t,y)dy-b(x)u(t,x)+f(x,u)\\
&\geq\int_{\mathbf{R}}K(x,y)u(t,y)dy-Bu(t,x).
\end{split}
\end{equation*}
Hence $\partial_{t}(e^{Bt}u(t,x))\geq e^{Bt}\int_{\mathbf{R}}K(x,y)u(t,y)dy>0$, i.e., $e^{Bt}u(t,x)$
is strictly increasing and
\begin{equation}\label{l4.3.1}
\begin{split}
e^{Bt}u(t,x)
&\geq I_{0}(x)+\int_{t_{0}}^{t}\int_{\mathbf{R}}K(x,y)e^{Bs}u(s,y)dyds,\\
\end{split}
\end{equation}
where $I_{0}(x)=e^{Bt_{0}}u(t_{0},x)$.
The monotonicity and (K3) yield that
\begin{equation*}
\begin{split}
e^{Bt}u(t,x)
&\geq I_{0}(x)+\int_{t_{0}}^{t}\int_{\mathbf{R}}K(x,y)e^{Bs}u(s,y)dyds\\
&\geq I_{0}(x)+\int_{t_{0}}^{t}\int_{\mathbf{R}}K(x,y)\chi_{\eta+\delta_{0}}(y)e^{Bt_{0}}u(t_{0},y)dyds\\
&\geq I_{0}(x)+e^{Bt_{0}}(t-t_{0})\int_{\mathbf{R}}K(x,y)\chi_{\eta+\delta_{0}}(y)u(t_{0},y)dy\\
&\geq I_{0}(x)+Ce^{Bt_{0}}(t-t_{0})\sigma(t_{0})
\end{split}
\end{equation*}
for $t\geq t_{0}$ and $|x|<\eta+\delta_{0}$, i.e.,
\begin{equation}\label{l4.3.2}
e^{Bt}u(t,x)\geq \chi_{\eta+\delta_{0}}(x)I_{1}(t,x)\ \forall\ t\geq t_{0}, x\in\mathbf{R},
\end{equation}
where $I_{1}(t,x)=I_{0}(x)+Ce^{Bt_{0}}(t-t_{0})\sigma(t_{0})$.

Using this, we can obtain from \eqref{l4.3.1} that
\begin{equation}\label{l4.3.4}
\begin{split}
e^{Bt}u(t,x)
&\geq I_{0}(x)+\int_{t_{0}}^{t}\int_{\mathbf{R}}K(x,y)e^{Bs}u(s,y)dyds\\
&\geq I_{0}(x)+\int_{t_{0}}^{t}\int_{\mathbf{R}}K(x,y)\chi_{\eta+\delta_{0}}(y)I_{1}(s,y)dyds\\
&=I_{0}(x)+\int_{t_{0}}^{t}\int_{\mathbf{R}}K(x,y)\chi_{\eta+\delta_{0}}(y)
 \big(I_{0}(y)+Ce^{Bt_{0}}(s-t_{0})\sigma(t_{0})\big)dyds\\
&=I_{0}(x)+\int_{t_{0}}^{t}\int_{\mathbf{R}}K(x,y)\chi_{\eta+\delta_{0}}(y)
 I_{0}(y)dy\\
&\quad+\int_{t_{0}}^{t}\int_{\mathbf{R}}K(x,y)\chi_{\eta+\delta_{0}}(y)Ce^{Bt_{0}}(s-t_{0})\sigma(t_{0})\big)dyds\\
&\geq I_{1}(t,x)+\sigma(t_{0})e^{Bt_{0}}C\frac{(t-t_{0})^{2}}{2}\int_{\mathbf{R}}K(x,y)\chi_{\eta+\delta_{0}}(y)dy\\
&\geq I_{1}(t,x)+\sigma(t_{0})e^{Bt_{0}}C^{2}\frac{(t-t_{0})^{2}}{2}\\
\end{split}
\end{equation}
for any $t\geq t_{0}, |x|<\eta+{\delta_{0}}$, i.e.,
\begin{equation*}
e^{Bt}u(t,x)\geq \chi_{\eta+\delta_{0}}(x)I_{2}(t,x)\ \forall\ t\geq t_{0}, x\in\mathbf{R},
\end{equation*}
where $I_{2}(t,x)=I_{1}(t,x)+\sigma(t_{0})e^{Bt_{0}}C^{2}\frac{(t-t_{0})^{2}}{2}.$

One can repeat this process to find that
\begin{equation*}
\begin{split}
e^{Bt}u(t,x)
&\geq e^{Bt_{0}}u(t_{0},x)+
\sigma(t_{0})e^{Bt_{0}}\sum_{n=1}^{\infty}\frac{(C(t-t_{0}))^{n}}{n!}\\
&=e^{Bt_{0}}\big(\sigma(t_{0})(e^{C(t-t_{0})}-1)+u(t_{0},x)\big)\ \forall\ t\geq t_{0}, |x|<\eta+{\delta_{0}}.
\end{split}
\end{equation*}
Therefore,
$$u(t,x)\geq e^{-B(t-t_{0})}\big(\sigma(t_{0})(e^{C(t-t_{0})}-1)+u(t_{0},x)\big)$$
for any $t\geq t_{0}, |x|<\eta+\delta_{0}$.
\end{proof}

\begin{proof}[Proof of Theorem \ref{thm4.1}]We only need to show that the theorem is valid when
$Q=[\tau,T]\times[-R,R]$.
We may assume that $u(0,x)\geq\sigma_{0}$ for $|x|\leq\eta_{0}$.
Otherwise, we can consider $u(\frac{\tau}{2},x)>0\ \forall x\in\mathbf{R}$ by Lemma \ref{lem4.2}.
First by Lemma \ref{lem4.3}
$$u(t,x)\geq e^{-Bt}\big(\sigma_{0}(e^{Ct}-1)+u(0,x)\big),\ \forall t\geq 0, |x|<\eta_{0}+\delta_{0}.$$
Let
$\sigma_{1}:=\min\limits_{|x|\leq\eta_{0}+\delta_{0}}u(t,x)\geq\sigma_{0}e^{-Bt}(e^{Ct}-1).$
Using Lemma \ref{lem4.3} again with $\eta=\eta_{0}+\delta_{0}$, we have
$$u(2t,x)\geq e^{-Bt}\big(\sigma_{1}(e^{Ct}-1)+u(t,x)\big),\ \forall t\geq 0, |x|<\eta_{0}+2\delta_{0}.$$
By induction, one can finally obtain that
$$u(nt,x)\geq e^{-Bt}\big(\sigma_{n-1}(e^{Ct}-1)+u((n-1)t,x)\big),\ \forall t\geq 0, |x|<\eta_{0}+n\delta_{0},\ n=1,2,\cdots,$$
where
$\sigma_{n}:=\min\limits_{|x|\leq\eta_{0}+n\delta_{0}}u(nt,x)\geq\sigma_{0}e^{-nBt}(e^{Ct}-1)^{n}.$
Therefore,
$$u(nt,x)\geq\sigma_{0}e^{-nBt}(e^{Ct}-1)^{n},\ \forall t\geq 0, |x|<\eta_{0}+n\delta_{0}, n=1,2,\cdots.$$
In particular,
$u(nt,nx)\geq\sigma_{0}e^{-nBt}(e^{Ct}-1)^{n},$ i.e.,
\begin{equation}\label{t4.1.1}
\frac{1}{n}\ln u(nt,nx)\geq\frac{1}{n}\ln\sigma_{0}-Bt+\ln(e^{Ct}-1)
\geq \ln\sigma_{0}-Bt+\ln(e^{Ct}-1)
\end{equation}
for $t\geq 0, |x|<\delta_{0}, n=1,2,\cdots.$ Now consider $t\in[\tau,T], |x|\leq R$.
Let $n_{\varepsilon}=({\frac{R}{\varepsilon\delta_{0}}+1})$,
$t_{\varepsilon}=\frac{t}{\varepsilon n_{\varepsilon}}$
and $x_{\varepsilon}=\frac{x}{\varepsilon n_{\varepsilon}}$. Then
$\varepsilon n_{\varepsilon}\in(\frac{R}{\delta_{0}},\frac{R}{\delta_{0}}+\varepsilon)$, and
$|x_{\varepsilon}|=\frac{|x|}{\varepsilon n_{\varepsilon}}\leq\frac{R}{\varepsilon n_{\varepsilon}}\leq \delta_{0}$. Using \eqref{t4.1.1}, we have
\begin{equation*}
\begin{split}
\varepsilon\ln u(\frac{t}{\varepsilon},\frac{x}{\varepsilon})
&=\varepsilon n_{\varepsilon}\frac{1}{n_{\varepsilon}}\ln u(n_{\varepsilon}t_{\varepsilon},n_{\varepsilon}x_{\varepsilon})\\
&\geq\varepsilon n_{\varepsilon}
 \big(\frac{1}{n_{\varepsilon}}\ln\sigma_{0}-Bt_{\varepsilon}+\ln(e^{Ct}-1)\big)\\
&\geq \varepsilon\ln\sigma_{0}-Bt+
 \varepsilon n_{\varepsilon}\ln(e^{Ct}-1)\\
&\geq -C
\end{split}
\end{equation*}
for some positive constant $C$ depends on $\tau,T,R$. Then we are done since
$z_{\varepsilon}(t,x)=\varepsilon\ln u(\frac{t}{\varepsilon},\frac{x}{\varepsilon})<0$.
\end{proof}

From Theorem \ref{thm4.1}, we know that
$$z_{\ast}(t,x)=\displaystyle{\liminf_{(s,y)\rightarrow(t,x),\varepsilon\to0}}z_{\varepsilon}(s,y)\leq0$$
is well defined on $(0,+\infty)\times\mathbf{R}$.
In the following content of this subsection we want to find a
Hamilton-Jacobi equation which is related to $z_{\ast}.$ Denote $B_{r}(t_{0},x_{0})=\{(t,x)|\ |t-t_{0}|^{2}+|x-x_{0}|^{2}<r^{2}\}$.

\begin{lemma}\label{lem4.4}
For any $(t_{0},x_{0})\in int\{z_{\ast}=0\}$, there exists $\theta>0$ such that $B_{\theta}(t_{0},x_{0})\in int\{z_{\ast}=0\}$ and
$\displaystyle{\liminf_{\varepsilon\to0}}\inf\limits_{(\tau,\xi)\in B_{\theta}(t_{0},x_{0})}v_{\varepsilon}(\tau,\xi)>0$.
\end{lemma}
\begin{proof}
For any $(t_{0},x_{0})\in \text{int}\{z_{\ast}=0\}$, we have
$z_{\ast}(t,x)=0$ for any $(t,x)\in B_{2\delta}(t_{0},x_{0})$ for some $\delta>0$.
Hence it is easy to find that $z_{\varepsilon}(t,x)\to0$ as $\varepsilon\to0$ uniformly in
$B_{\delta}(t_{0},x_{0})$. Now for any $(\tau,\xi)\in B_{\frac{\delta}{4}}(t_{0},x_{0})$, let $\phi(t,x;\tau,\xi)=-(|t-\tau|^{2}+|x-\xi|^{2})$. For any $\eta\in(0,\frac{\delta^{2}}{16})$, there exists
$\varepsilon_{0}$ (only depending on $\eta$ and $z_{\varepsilon}$) such that $-\eta<z_{\varepsilon}(t,x)\leq0$
for all $(t,x)\in B_{\delta}(t_{0},x_{0})$ and $0<\varepsilon\leq\varepsilon_{0}$. Therefore,
$$\big(z_{\varepsilon}(t,x)-\phi(t,x;\tau,\xi)\big)
>\big(|t-\tau|^{2}+|x-\xi|^{2}\big)-\eta\geq\frac{\delta^{2}}{2}>0\ \ \forall\ (t,x)\in\partial B_{\delta}(t_{0},x_{0}),$$
$$z_{\varepsilon}(t,x)-\phi(t,x;\tau,\xi)\big|_{(t,x)=(\tau,\xi)}=z_{\varepsilon}(\tau,\xi)\leq0.$$
That is to say, $z_{\varepsilon}(t,x)-\phi(t,x;\tau,\xi)$ reaches
its minimum at some point, say $(t_{\varepsilon}(\tau,\xi),x_{\varepsilon}(\tau,\xi))$, over $B_{\delta}(t_{0},x_{0})$. Note that for any $r\in(\sqrt{\eta},\delta], \varepsilon\in(0,\varepsilon_{0})$. We have
$$\big(z_{\varepsilon}(t,x)-\phi(t,x;\tau,\xi)\big)>0
\geq\min\limits_{B_{\delta}(t_{0},x_{0})}\big(z_{\varepsilon}(t,x)-\phi(t,x;\tau,\xi)\big)\ \ \forall(t,x)\in B_{\delta_{0}}(t_{0},x_{0})\setminus B_{r}(\tau,\xi),$$
which means that
$(t_{\varepsilon}(\tau,\xi),x_{\varepsilon}(\tau,\xi))\in\overline{B_{\sqrt{\eta}}(\tau,\xi)}$, i.e.,
$$|t_{\varepsilon}(\tau,\xi)-\tau|^{2}+|x_{\varepsilon}(\tau,\xi)-\xi|^{2}\leq\eta,\ \forall(\tau,\xi)\in B_{\frac{\delta}{4}}(t_{0},x_{0}).$$
From this, one can easily find that $(t_{\varepsilon}(\tau,\xi),x_{\varepsilon}(\tau,\xi))\to(\tau,\xi)$
as $\varepsilon\to0$ uniformly for $(\tau,\xi)\in B_{\frac{\delta}{4}}(t_{0},x_{0})$.
We will write $t_{\varepsilon}(\tau,\xi), x_{\varepsilon}(\tau,\xi))$ and $\phi(t,x;\tau,\xi)$ by
$t_{\varepsilon}, x_{\varepsilon}$ and $\phi(t,x)$ for simplicity. Obviously, we have
\begin{equation}\label{l4.4.1}
\partial_{t}z_{\varepsilon}(t_{\varepsilon}, x_{\varepsilon})-\partial_{t}\phi(t_{\varepsilon}, x_{\varepsilon})=0,
\end{equation}
\begin{equation}\label{l4.4.2}
z_{\varepsilon}(t,y)-\phi(t,y)\geq z_{\varepsilon}(t_{\varepsilon}, x_{\varepsilon})-\phi(t_{\varepsilon}, x_{\varepsilon}),\ \forall(t,y)\in B_{\delta}(t_{0},x_{0}).
\end{equation}
Then at $(t_{\varepsilon},x_{\varepsilon})$, we have
\begin{equation}\label{l4.4.3}
\begin{split}
\partial_{t} z_{\varepsilon}(t,x)
&\geq\frac{1}{\varepsilon}\int_{B_{\delta}(x_{0})}K_{\varepsilon}(x,y)\exp\big(\frac{z_{\varepsilon}(t,y)-z_{\varepsilon}(t,x)}{\varepsilon}\big)dy
-b_{\varepsilon}(x)+\frac{f_{\varepsilon}(x,v_{\varepsilon})}{v_{\varepsilon}}\\
&\geq\frac{1}{\varepsilon}\int_{B_{\delta}(x_{0})}K_{\varepsilon}(x,y)
\exp\big(\frac{\phi(t,y)-\phi(t,x)}{\varepsilon}\big)dy
-b_{\varepsilon}(x)+\frac{f_{\varepsilon}(x,v_{\varepsilon})}{v_{\varepsilon}}\\
&\geq\frac{1}{\varepsilon}\int_{B_{\frac{\delta}{2}}(x_{\varepsilon})}K_{\varepsilon}(x,y)
\exp\big(\frac{\phi(t,y)-\phi(t,x)}{\varepsilon}\big)dy
-b_{\varepsilon}(x)+\frac{f_{\varepsilon}(x,v_{\varepsilon})}{v_{\varepsilon}},
\end{split}
\end{equation}
since $B_{\delta}(x_{0})\supset B_{\frac{\delta}{2}}(x_{\varepsilon})$ for $\varepsilon$ small.
Combining this with \eqref{l4.4.1}, we have
\begin{equation}\label{l4.4.4}
0<\frac{f_{\varepsilon}(x_{\varepsilon},v_{\varepsilon})}{v_{\varepsilon}}\leq\partial_{t}\phi(t_{\varepsilon},x_{\varepsilon})
+b_{\varepsilon}(x_{\varepsilon})-
\frac{1}{\varepsilon}\int_{B_{\frac{\delta}{2}}(x_{\varepsilon})}K_{\varepsilon}(x_{\varepsilon},y)
\exp\big(\frac{\phi(t_{\varepsilon},y)-\phi(t_{\varepsilon},x_{\varepsilon})}{\varepsilon}\big)dy.
\end{equation}

Claim: $F_{\varepsilon}(x):=b_{\varepsilon}(x_{\varepsilon})-
\frac{1}{\varepsilon}\int_{B_{\frac{\delta}{2}}(x_{\varepsilon})}K_{\varepsilon}(x_{\varepsilon},y)
\exp\big(\frac{\phi(t_{\varepsilon},y)-\phi(t_{\varepsilon},x_{\varepsilon})}{\varepsilon}\big)dy\to0$
as $\varepsilon\to0.$

Proof of Claim:
\begin{equation}\label{l4.4.5}
\begin{split}
|F_{\varepsilon}(x)|
&=|\frac{1}{\varepsilon}\int_{\mathbf{R}}K_{\varepsilon}(x_{\varepsilon},y)dy
-\frac{1}{\varepsilon}\int_{B_{\frac{\delta}{2}}(x_{\varepsilon})}K_{\varepsilon}(x_{\varepsilon},y)
\exp\big(\frac{(x_{\varepsilon}-y)(x_{\varepsilon}-\xi+y-\xi)}{\varepsilon}\big)dy|\\
&=|\int_{\mathbf{R}}K(\frac{x_{\varepsilon}}{\varepsilon},\frac{x_{\varepsilon}}{\varepsilon}-y)dy-
\int_{B_{\frac{\delta}{2\varepsilon}}(0)}K(\frac{x_{\varepsilon}}{\varepsilon},\frac{x_{\varepsilon}}{\varepsilon}-y)
\exp\big(2y(x_{\varepsilon}-\xi)-\varepsilon y^{2}\big)dy|.
\end{split}
\end{equation}
For any $\sigma>0$, there exists $R>0$ such that
\begin{equation}
\int_{B_{R}^{C}(0)}K(\frac{x_{\varepsilon}}{\varepsilon},\frac{x_{\varepsilon}}{\varepsilon}-y)dy<\sigma,
\end{equation}
and
\begin{equation}
\int_{B_{R}^{C}(0)}K(\frac{x_{\varepsilon}}{\varepsilon},\frac{x_{\varepsilon}}{\varepsilon}-y)
\exp\big(2y(x_{\varepsilon}-\xi)-\varepsilon y^{2}\big)dy
\leq\int_{B_{R}^{C}(0)}K(\frac{x_{\varepsilon}}{\varepsilon},\frac{x_{\varepsilon}}{\varepsilon}-y)
\exp(\frac{\delta y}{2})dy\leq\sigma
\end{equation}
for any $\varepsilon\in(0,\varepsilon_{0})$ since $K(x,x-\cdot)e^{-p\cdot}$ is uniformly integrable w.r.t. $x$. Then
\begin{equation}\label{l4.4.6}
\begin{split}
|F_{\varepsilon}(x)|
&\leq|\int_{B_{R}(0)}K(\frac{x_{\varepsilon}}{\varepsilon},\frac{x_{\varepsilon}}{\varepsilon}-y)-
K(\frac{x_{\varepsilon}}{\varepsilon},\frac{x_{\varepsilon}}{\varepsilon}-y)
\exp\big(2y(x_{\varepsilon}-\xi)-\varepsilon y^{2}\big)dy|+2\sigma\\
&=\int_{B_{R}(0)}K(\frac{x_{\varepsilon}}{\varepsilon},\frac{x_{\varepsilon}}{\varepsilon}-y)dy
\sup\limits_{y\in B_{R}(0)}|
\big(1-\exp\big(2y(x_{\varepsilon}-\xi)-\varepsilon y^{2}\big)\big)|+2\sigma\\
&\leq3\sigma
\end{split}
\end{equation}
for $\varepsilon$ small enough since $\exp\big(2y(x_{\varepsilon}-\xi)-\varepsilon y^{2}\big)\to1$ as
$\varepsilon\to0$ uniformly in $y\in B_{R}(0)$. Thus the proof of claim is finished.

Note that $\partial_{t}\phi(t_{\varepsilon},x_{\varepsilon})=2(t_{\varepsilon}-\tau)\to0$ as $\varepsilon\to0$.
Then from \eqref{l4.4.4}, we have $0<\frac{f_{\varepsilon}(x_{\varepsilon},v_{\varepsilon})}{v_{\varepsilon}}\leq o(1)$ as $\varepsilon\to0$ uniformly for $(\tau,\xi)\in B_{\frac{\delta}{4}}(t_{0},x_{0})$.
On the other hand, $f(x,\cdot)\in\mathcal{C}^{1+\gamma}([0,1])$
uniformly with respect to $x\in\mathbf{R}$ yields that there exists $C>0$ such that
\begin{equation}\label{l4.4.7}
\frac{f(x,s)}{s}
\geq f^{\prime}_{s}(x,0)-Cs^{\gamma}\ \forall x\in\mathbf{R}, s\in[0,1],
\end{equation}
which yields
$$0<f^{\prime}_{s}(\frac{x_{\varepsilon}}{\varepsilon},0)\leq Cv^{\gamma}_{\varepsilon}+
\frac{f_{\varepsilon}(x_{\varepsilon},v_{\varepsilon})}{v_{\varepsilon}}
\leq Cv^{\gamma}_{\varepsilon}+o(1)\ \text{as}\ \varepsilon\to 0.$$
Then we have
$${C}^{\frac{1}{\gamma}}\liminf\limits_{\varepsilon\to0}v_{\varepsilon}(t_{\varepsilon},x_{\varepsilon})
\geq{(\inf\limits_{x\in\mathbf{R}}f^{\prime}_{s}(x,0))}^{\frac{1}{\gamma}}>0,$$
where the last inequality follows from $0<\inf \limits_{x\in\mathbf{R}}f(x,s)\leq f(x,s)\leq f_{s}^{\prime}(x,0)s$ for any $s\in(0,1)$.
Furthermore, by the definition of $(t_{\varepsilon},x_{\varepsilon})$, we have
$$z_{\varepsilon}(t_{\varepsilon},x_{\varepsilon})\leq
z_{\varepsilon}(t_{\varepsilon},x_{\varepsilon})-\phi(t_{\varepsilon},x_{\varepsilon})\leq
z_{\varepsilon}(\tau,\xi)-\phi(\tau,\xi)=z_{\varepsilon}(\tau,\xi),$$ which yields that
$v_{\varepsilon}(\tau,\xi)\geq v_{\varepsilon}(t_{\varepsilon}(\tau,\xi),x_{\varepsilon}(\tau,\xi))$
$\forall(\tau,\xi)\in B_{\frac{\delta}{4}}(t_{0},x_{0})$. Thus
$$\displaystyle{\liminf_{\varepsilon\to0}}\inf\limits_{(\tau,\xi)\in B_{\theta}(t_{0},x_{0})}v_{\varepsilon}(\tau,\xi)\geq
\liminf\limits_{\varepsilon\to0}v_{\varepsilon}(t_{\varepsilon},x_{\varepsilon})>0.$$
\end{proof}

\begin{lemma}\label{lem4.5}
The lower semi-continuous function $z_{\ast}$ is a viscosity supersolution of
$$\max\{\partial_{t}z_{\ast}-\underline{H}(\partial_{x}z_{\ast}),z_{\ast}\}\geq0\ \ \ \
(t,x)\in(0,+\infty)\times(0,+\infty).$$
\end{lemma}
\begin{proof}
Note that $z_{\ast}\leq0$, hence we only need to show that
$$\partial_{t}z_{\ast}(\tau,x)-\underline{H}(\partial_{x}z_{\ast}(\tau,x))\geq0\ \ \
(t,x)\in\{(t,x)\in(0,+\infty)\times(0,+\infty)|z_{\ast}(t,x)<0\}$$
in the sense of viscosity solution.
For a smooth function $\phi$ defined on $(0,+\infty)\times(0,+\infty)$,
assume that $z_{\ast}-\phi$ reaches its strict minimum at $(t_{0},x_{0})$ over $\overline{B_{\delta}(t_{0},x_{0})}$, with $z_{\ast}(t_{0},x_{0})<0$.
Then we need to show that
$$\partial_{t}\phi(t_{0},x_{0})-\underline{H}(\partial_{x}\phi(t_{0},x_{0}))\geq0.$$
Denote $p:=\partial_{x}\phi(t_{0},x_{0})$. Fix some $R\in\mathbf{R}$ large enough. For any $\mu>0$ small, there exists $\psi\in\mathcal{A}$ such that
\begin{equation}\label{l4.5.1}
L_{p}\psi\geq(\underline{{\lambda}_{1}}(p,R)-\mu)\psi\ \text{on}\ I_{R}.
\end{equation}
Let $\beta_{\varepsilon}(x):=\varepsilon\ln\psi_{\varepsilon}(x)$, where $\psi_{\varepsilon}(x)=\psi(\frac{x}{\varepsilon})$.
Then $\beta_{\varepsilon}(x)$ satisfies
\begin{equation}\label{l4.5.2}
\frac{1}{\varepsilon}\int_{\mathbf{R}}
K_{\varepsilon}(x,y)\exp\big(\frac{\beta_{\varepsilon}(y)-\beta_{\varepsilon}(x)+p(y-x)}{\varepsilon}\big)dy
-b_{\varepsilon}(x)+f_{s}^{\prime}(\frac{x}{\varepsilon},0)\geq \underline{{\lambda}_{1}}(p,R)-\mu
\end{equation}
on $I_{\varepsilon R}$ by \eqref{l4.5.1}. Moreover, $\beta_{\varepsilon}(x)\to0$ as $\varepsilon\to0$
locally uniform with respect to $x$ since $\psi\in\mathcal{A}$. After a similar argument to
\cite[Proposition 4.3]{B2}, there exist $\varepsilon_{n}$, $(t_{n},x_{n})\in B_{\delta}(t_{0},x_{0})$
such that $(\varepsilon_{n},t_{n},x_{n})\to(0,t_{0},x_{0})$
and $z_{\varepsilon_{n}}(t_{n},x_{n})\to z_{\ast}(t_{0},x_{0})$ as $n\to+\infty$, and $z_{\varepsilon_{n}}-\phi-\varepsilon_{n}\ln\psi_{\varepsilon_{n}}$ reaches its minimum at
$(t_{n},x_{n})$ over $\overline{B_{\delta}(t_{0},x_{0})}$ for $n\geq n_{0}$, i.e.,
$$z_{\varepsilon_{n}}(t,x)-\phi(t,x)-\beta_{\varepsilon_{n}}(x)\geq
z_{\varepsilon_{n}}(t_{n},x_{n})-\phi(t_{n},x_{n})-\beta_{\varepsilon_{n}}(x_{n})\ \
\forall (t,x)\in\overline{B_{\delta}(t_{0},x_{0})}.$$
Hence one can obtain
\begin{equation}\label{l4.5.3}
\begin{split}
\partial_{t}\phi(t_{n},x_{n})
&=\partial_{t}z_{\varepsilon_{n}}(t_{n},x_{n})\\
&=\frac{1}{\varepsilon_{n}}\int_{\mathbf{R}}
K_{\varepsilon_{n}}(x_{n},y)\exp\big(\frac{z_{\varepsilon_{n}}(t_{n},y)-z_{\varepsilon_{n}}(t_{n},x_{n})}{\varepsilon_{n}}\big)dy
-b_{\varepsilon_{n}}(x_{n})+\frac{f(\frac{x_{n}}{\varepsilon_{n}},v_{\varepsilon_{n}})}{v_{\varepsilon_{n}}}\\
&\geq\frac{1}{\varepsilon_{n}}\int_{B_{\delta}(x_{0})}
K_{\varepsilon_{n}}(x_{n},y)\exp\big(\frac{z_{\varepsilon_{n}}(t_{n},y)-z_{\varepsilon_{n}}(t_{n},x_{n})}{\varepsilon_{n}}\big)dy
-b_{\varepsilon_{n}}(x_{n})+\frac{f(\frac{x_{n}}{\varepsilon_{n}},v_{\varepsilon_{n}})}{v_{\varepsilon_{n}}}\\
&\geq\frac{1}{\varepsilon_{n}}\int_{B_{\delta}(x_{0})}
K_{\varepsilon_{n}}(x_{n},y)\exp\big(\frac{\phi(t_{n},y)+\beta_{\varepsilon_{n}}(y_{n})
 -\phi(t_{n},x_{n})-\beta_{\varepsilon_{n}}(x_{n})}{\varepsilon_{n}}\big)dy\\
&\quad-b_{\varepsilon_{n}}(x_{n})+\frac{f(\frac{x_{n}}{\varepsilon_{n}},v_{\varepsilon_{n}})}{v_{\varepsilon_{n}}}.
\end{split}
\end{equation}
Note that $x_{n}\in I_{\varepsilon R}$ and $B_{\delta}(x_{0})\supset B_{\frac{\delta}{2}}(x_{n})$
for $n$ large enough, hence \eqref{l4.5.2} and \eqref{l4.5.3} yield that
\begin{equation}\label{l4.5.4}
\begin{split}
\partial_{t}\phi(t_{n},x_{n})
&\geq\frac{1}{\varepsilon_{n}}\int_{B_{\frac{\delta}{2}}(x_{n})}
K_{\varepsilon_{n}}(x_{n},y)\exp\big(\frac{\phi(t_{n},y)+\beta_{\varepsilon_{n}}(y_{n})
 -\phi(t_{n},x_{n})-\beta_{\varepsilon_{n}}(x_{n})}{\varepsilon_{n}}\big)dy\\
&\quad-\frac{1}{\varepsilon_{n}}\int_{\mathbf{R}}
K_{\varepsilon_{n}}(x_{n},y)\exp\big(\frac{\beta_{\varepsilon_{n}}(y)-\beta_{\varepsilon}(x_{n})+p(y-x_{n})}{\varepsilon_{n}}\big)dy
+\underline{{\lambda}_{1}}(p,N)-\mu\\
&\quad+\frac{f(\frac{x_{n}}{\varepsilon_{n}},v_{\varepsilon_{n}})}{v_{\varepsilon_{n}}}
-f_{s}^{\prime}(\frac{x_{n}}{\varepsilon_{n}},0)\\
&\geq\int_{B_{\frac{\delta}{2\varepsilon_{n}}}(0)}K(\frac{x_{n}}{\varepsilon_{n}},\frac{x_{n}}{\varepsilon_{n}}-y)
\exp\big(\frac{\phi(t_{n},x_{n}-\varepsilon_{n}y)-\phi(t_{n},x_{n})}{\varepsilon_{n}}
+\frac{\beta_{\varepsilon_{n}}(x_{n}-\varepsilon_{n}y)
-\beta_{\varepsilon_{n}}(x_{n})}{\varepsilon_{n}}\big)dy\\
&\quad-\int_{\mathbf{R}}K(\frac{x_{n}}{\varepsilon_{n}},\frac{x_{n}}{\varepsilon_{n}}-y)
\exp\big(\frac{\beta_{\varepsilon_{n}}(x_{n}-\varepsilon_{n}y)
-\beta_{\varepsilon_{n}}(x_{n})}{\varepsilon_{n}}-py\big)dy
+\underline{{\lambda}_{1}}(p,N)-\mu\\
&\quad+\frac{f(\frac{x_{n}}{\varepsilon_{n}},v_{\varepsilon_{n}})}{v_{\varepsilon_{n}}}
-f_{s}^{\prime}(\frac{x_{n}}{\varepsilon_{n}},0).
\end{split}
\end{equation}
Note also that
$\exp\big(\frac{\beta_{\varepsilon_{n}}(x_{n}-\varepsilon_{n}y)-\beta_{\varepsilon_{n}}(x_{n})}{\varepsilon_{n}}\big)
\leq\frac{\sup\limits_{x\in\mathbf{R}}\psi(x)}{\inf\limits_{x\in\mathbf{R}}\psi(x)}$,
and there exists some constant $M$ such that
$|\frac{\phi(t_{n},x_{n}-\varepsilon_{n}y)-\phi(t_{n},x_{n})}{\varepsilon_{n}}|\leq My$
for any $n\geq n_{0}, y\in B_{\frac{\delta}{2\varepsilon_{n}}}(0)$. Moreover,
since $\lim\limits_{n\to+\infty}z_{\varepsilon_{n}}(t_{n},x_{n})=z_{\ast}(t_{0},x_{0})<0$, we have
$v_{\varepsilon_{n}}(t_{n},x_{n})=\exp{\frac{z_{\varepsilon_{n}}(t_{n},x_{n})}{\varepsilon_{n}}}\to0$
as $n\to+\infty$. Then
$\frac{f(\frac{x_{n}}{\varepsilon_{n}},v_{\varepsilon_{n}})}{v_{\varepsilon_{n}}}
-f^{\prime}_{s}(\frac{x_{n}}{\varepsilon_{n}},0)\to0$
as $n\to+\infty$. For any $\sigma>0$, combining these and the uniformly integrability of $K(x,x-\cdot)e^{-p\cdot}$ with
\eqref{l4.5.4}, there exists $r$ such that
\begin{equation}\label{l4.5.5}
\begin{split}
\partial_{t}\phi(t_{n},x_{n})
&\geq-\frac{\sup\limits_{x\in\mathbf{R}}\psi(x)}{\inf\limits_{x\in\mathbf{R}}\psi(x)}\int_{B_{r}(0)}
K(\frac{x_{n}}{\varepsilon_{n}},\frac{x_{n}}{\varepsilon_{n}}-y)dy
\sup\limits_{B_{r}(0)}\big|\exp\big(\frac{\phi(t_{n},x_{n}-\varepsilon_{n}y)-\phi(t_{n},x_{n})}{\varepsilon_{n}}\big)-e^{-py}\big|\\
&\quad+2\sigma+\underline{{\lambda}_{1}}(p,N)-\mu+o(1).
\end{split}
\end{equation}
Now from the definition of $p$, it is no difficulty to find that
$\exp\big(\frac{\phi(t_{n},x_{n}-\varepsilon_{n}y)-\phi(t_{n},x_{n})}{\varepsilon_{n}}\big)-e^{-py}\to0$
locally uniform in y. Taking $n\to\infty$ in \eqref{l4.5.5}, we have
$$\partial_{t}\phi(t_{0},x_{0})\geq\underline{{\lambda}_{1}}(p,R)-\mu.$$
Then taking $\mu\to0^{+}$, we have
$$\partial_{t}\phi(t_{0},x_{0})-\underline{{\lambda}_{1}}(p,R)\geq0.$$
Finally, taking $R\to+\infty$, we obtain
$$\partial_{t}\phi(t_{0},x_{0})-\underline{H}(p)\geq0.$$
Thus complete the proof.
\end{proof}

Next we need consider the convex conjugate of $\underline{H}$ which is given by
$\underline{H}^{*}(q):=\sup\limits_{p\in\mathbf{R}}(pq-\underline{H}(p))\geq pq-\underline{H}(p)$
for any $p\in\mathbf{R}$. It is well defined by Proposition \ref{prop2.2}.
Then we have the following estimate for $z_{\ast}:$
\begin{lemma}\label{lem4.6}
One has $z_{\ast}(t,x)\geq\min\{-t\underline{H}^{*}(-\frac{x}{t}),0\}$ for all $(t,x)\in(0,+\infty)\times(0,+\infty).$
\end{lemma}
\begin{proof}
See\cite[Lemma 4.4]{B2} or \cite[Lemma 4.5]{LZh}.
\end{proof}

\subsection{Complete the proof of Theorem \ref{thm2.1}}

For any positive number $a,b$, let
$$S_{t}(a,b):=\{x\in\mathbf{R}|-at<x<bt\}.$$
We first prove a lemma we will need later:
\begin{lemma}\label{lem4.7}
Let $\tau>0$, and assume that $0\leq v, \tilde{v}\leq 1$ satisfy
\begin{equation}\label{l4.7.1}
u_{t}(t,x)=\int_{\mathbf{R}}K(x,y)u(t,y)dy-b(x)u(t,x)+f(x,u)\ \ t>0,x\in\mathbf{R},
\end{equation}
and $v(0,x)=v_{0}(x), \tilde{v}(0,x)=\tilde{v}_{0}(x)$ with $v_{0}(x)=\tilde{v}_{0}(x)$
for $x\in S_{\tau}(a,b)$. Then
$$\big|v(t,z)-\tilde{v}(t,z)\big|\leq ce^{ct}e^{-\theta\tau},\ \forall t\geq0,\ z\in S_{\tau}(a-\theta,b-\theta),$$
where $0<\theta<\min\{a,b\}$ and $c$ is a constant depending on $K, f$. In particular, fixing $T_{1}>0$,
for any $\sigma>0$, one can find $T$ such that
$$\big|v(T_{1},z)-\tilde{v}(T_{1},z)\big|\leq\sigma,\ \forall\tau>T,\ z\in S_{\tau}(a-\theta,b-\theta).$$
\end{lemma}
\begin{proof}
First by the semigroup theory, there exists $c_{0}\in\mathbf{R}$ such that
\begin{equation}\label{l4.7.2}
\|e^{Kt}\phi\|_{z}\leq M_{0}e^{c_{0}t}\|\phi\|_{z},\ \forall t\geq0, z\in\mathbf{R}, \phi\in X_{z}.
\end{equation}
Let $w=v-\tilde{v}$. Then $w$ satisfies
\begin{equation*}
  \left\{
   \begin{aligned}
   w_{t}(t,x)=Kw(t,x)+(\overline{c}(t,x)-b(x))w(t,x)\ \ &t>0,x\in\mathbf{R},  \\
   w(0,x)=v_{0}(x)-\tilde{v}_{0}(x),\ \  &x\in\mathbf{R},\\
   \end{aligned}
   \right.
  \end{equation*}
where $\overline{c}(t,x)=\int_{0}^{1}f_{s}^{\prime}(x,\tilde{v}(t,x)+s(v(t,x)-\tilde{v}(t,x)))ds$.
Hence
\begin{equation}\label{l4.7.3}
 w(t,x)=M_{0}e^{Kt}w(0,x)+M_{0}\int_{0}^{t}e^{K(t-s)}(\overline{c}(s,x)-b(x))w(s,x)ds,
\end{equation}
and $|\overline{c}(t,x)-b(x)|\leq M\ \forall (t,x)\in(0,+\infty)\times\mathbf{R}$ for some $M$.
Using \eqref{l4.7.2}, we have
$$\|w(t,\cdot)\|_{z}\leq M_{0}e^{c_{0}t}\|w(0,\cdot)\|_{z}+M_{0}M\int_{0}^{t}e^{c_{0}(t-s)}\|w(s,\cdot)\|_{z}ds.$$
The Gronwall's inequality yields that
$$\|w(t,\cdot)\|_{z}\leq M_{0}e^{(c_{0}+M_{0}M)t}\|w(0,\cdot)\|_{z}\ \forall t\geq0, z\in\mathbf{R},$$
i.e., $\sup\limits_{x\in\mathbf{R}}\big(e^{-(x-z)}|v(t,x)-\tilde{v}(t,x)|\big)\leq
M_{0}e^{(c_{0}+M_{0}M)t}\sup\limits_{x\in\mathbf{R}}\big(e^{-(x-z)}|v(0,x)-\tilde{v}(0,x)|\big)\ \forall t\geq0, z\in\mathbf{R}.$
In particular, for any $z\in\ S_{\tau}(a-\theta,b-\theta)$, we have
\begin{equation}\label{l4.7.4}
\begin{split}
|v(t,z)-\tilde{v}(t,z)|
&\leq M_{0}e^{(c_{0}+M_{0}M)t}\sup\limits_{x\in\mathbf{R}}\big(e^{-(x-z)}|v(0,x)-\tilde{v}(0,x)|\big)\\
&=M_{0}e^{(c_{0}+M_{0}M)t}\sup\limits_{x\in\mathbf{R}\setminus S_{\tau}(a,b)}\big(e^{-(x-z)}|v(0,x)-\tilde{v}(0,x)|\big)\\
&\leq2M_{0}e^{(c_{0}+M_{0}M)t}\sup\limits_{x\in\mathbf{R}\setminus S_{\tau}(a,b)}e^{-(x-z)}\\
&\leq2M_{0}e^{(c_{0}+M_{0}M)t}e^{-\theta\tau}.
\end{split}
\end{equation}
\end{proof}
\begin{proof}[Proof of part 2 of Theorem \ref{thm2.1}] We will prove it in three steps.\\
Step 1: For any $\omega\in(0,\underline{\omega}), \omega^{-}\in(0,\underline{\omega}^{-})$, we have
$(1,\omega)$ and $(1,\omega^{-})\in \text{int}\{z_{\ast}=0\}$.

By the definition of $\underline{\omega},\ \exists\ \varepsilon>0$ such that
$\underline{H}(-p)\geq p\omega(1+\varepsilon)$ for any $p>0$; also from Proposition \ref{prop2.2} one can find
that there exists $0<\eta\leq\underline{H}(0)$ such that
$\underline{H}(-p)\geq p\omega+\eta$,
i.e., $-\eta\geq(-p)(-\omega)-\underline{H}(-p)$ for all
$p\in\mathbf{R}$. Then we obtain
$-\underline{H}^{*}(-\omega)\geq\eta>0$. Hence by the continuity of $\underline{H}^{*}$ and
Lemma \ref{lem4.6}, there exists a neighbourhood $B(1,\omega)$ of $(1,\omega)\in(0,+\infty)\times(0,+\infty)$ such that for any
$(t,x)\in B(1,\omega)$. Then we have
$$z_{\ast}(t,x)\geq \min\{-t\underline{H}^{*}(-\frac{x}{t}),0\}=0,$$
that is to say, $(1,\omega)\in \text{int}\{z_{\ast}=0\}$. Similarly, $(1,\omega^{-})\in \text{int}\{z_{\ast}=0\}$.

Step 2: Show that $\liminf\limits_{t\to+\infty}\{\inf\limits_{x\in S_{t}(\omega^{-},\omega)}u(t,x)\}>0.$

Note that $(1,\omega)$ and $(1,\omega^{-})\in \text{int}\{z_{\ast}=0\}$. Then by Lemma \ref{lem4.4}, there exists
$\theta>0$ such that
\begin{equation}\label{step2.1}
\liminf\limits_{t\to+\infty}\{\inf\limits_{\tilde{\omega}\in\Omega} u(t,\tilde{\omega}t)\}>0,
\end{equation}
where $\Omega:=(-\omega^{-},-(\omega^{-}-\theta))\cup(\omega-\theta,\omega)$.
Suppose that $\liminf\limits_{t\to+\infty}\{\inf\limits_{x\in S_{t}(\omega^{-},\omega)}u(t,x)\}=0$.
Then there exists $(t_{n},x_{n})$ satisfying $\lim\limits_{n\to\infty}u(t_{n},x_{n})=0$, with
$x_{n}\in S_{t_{n}}(-(\omega^{-}-\theta),\omega-\theta)$ since \eqref{step2.1}
(hence $B_{\theta}(x_{n})\subset S_{t_{n}}(\omega^{-},\omega)$) , and
$u(t_{n},x_{n})=\inf\limits_{t\in[0,t_{n}], x\in S_{t_{n}}(\omega^{-},\omega)}u(t,x)$. Therefore, at
$(t_{n},x_{n})$, we have
\begin{equation}\label{step2.2}
\begin{split}
0\geq\partial_{t}u(t_{n},x_{n})
&=\int_{\mathbf{R}}K(x_{n},y)u(t_{n},y)dy-b(x_{n})u(t_{n},x_{n})+f(x_{n},u(t_{n},x_{n}))\\
&\geq\int_{B_{\theta t_{n}}(x_{n})}K(x_{n},y)u(t_{n},y)dy-b(x_{n})u(t_{n},x_{n})+f(x_{n},u(t_{n},x_{n}))\\
&\geq\int_{B_{\theta t_{n}}(x_{n})}K(x_{n},y)u(t_{n},x_{n})dy-b(x_{n})u(t_{n},x_{n})+f(x_{n},u(t_{n},x_{n}))\\
&=\int_{B_{\theta t_{n}}^{C}(x_{n})}K(x_{n},y)u(t_{n},x_{n})dy+f(x_{n},u(t_{n},x_{n})).
\end{split}
\end{equation}
Combining this with \eqref{l4.4.7}
$$0\geq\limsup\limits_{n\to\infty}\Big(
\int_{B_{\theta t_{n}}^{C}(x_{n})}K(x_{n},y)dy+\frac{f(x_{n},u(t_{n},x_{n}))}{u(t_{n},x_{n})}\Big)\geq
\limsup\limits_{n\to\infty}f^{\prime}_{s}(x_{n},0)\geq\inf\limits_{x\in\mathbf{R}}f^{\prime}_{s}(x,0)
>0,$$
which is a contradiction.
Hence $\liminf\limits_{t\to+\infty}\{\inf\limits_{x\in S_{t}(\omega^{-},\omega)}u(t,x)\}>0.$\\
Step 3: End the proof.

We only need to show that
$\liminf\limits_{t\to+\infty}\{\inf\limits_{x\in S_{t}(\omega^{-}-2\theta,\omega-2\theta)}u(t,x)\}=1$
for any $\theta>0$ since $\omega\in(0,\underline{\omega}), \omega^{-}\in(0,\underline{\omega}^{-})$ are arbitrary.

Let $v$ be a solution of \eqref{c4.1.1} with initial value $v_{0}=\alpha$, where
$\alpha:=\frac{1}{2}\liminf\limits_{t\to+\infty}\{\inf\limits_{x\in S_{t}(\omega^{-},\omega)}u(t,x)\}\in(0,1).$
Then for any $\sigma>0$, by Corollary \ref{cor4.1}, there exists $T_{1}$ such that
$v(t,x)\geq1-\sigma\ \forall t\geq T_{1}, x\in\mathbf{R}$.
Let $\underline{u}^{\tau}(t,x)$ be a solution of \eqref{l4.7.1} with initial value $\underline{u}^{\tau}(0,x)=\min\limits_{x\in\mathbf{R}}\{\alpha, u(\tau,x)\}$.
Then from Step 2 one can easily find that $\underline{u}^{\tau}(0,x)=\alpha$ for
$x\in S_{\tau}(\omega^{-},\omega)$ whenever $\tau$ is large, say $\tau>T_{2}$ for some $T_{2}$.
Using Lemma \ref{lem4.7} by replacing $\tilde{v}, a, b$ with $\underline{u}^{\tau}, \omega^{-}, \omega$ respectively, one can find a constant $T_{3}(\geq T_{2})$ such that
$\big|v(T_{1},z)-\underline{u}^{\tau}(T_{1},z)\big|\leq\sigma,\ \forall\tau>T,\ z\in S_{\tau}(\omega^{-}-\theta,\omega-\theta).$ Therefore, $\underline{u}^{\tau}(T_{1},z)\geq v(T_{1},z)-\sigma\geq1-2\sigma,\ \forall\tau>T,\ z\in S_{\tau}(\omega^{-}-\theta,\omega-\theta).$
Moreover, there exists $T_{4}>0$ such that $S_{\tau}(\omega^{-}-2\theta,\omega-2\theta)\subset S_{\tau-T_{1}}(\omega^{-}-\theta,\omega-\theta)$ for $\tau\geq T_{4}$.
Now taking $T_{0}=T_{1}+T_{3}+T_{4}$, we have
\begin{equation}\label{step3.1}
\underline{u}^{t-T_{1}}(T_{1},z)\geq 1-2\sigma,\ \forall t>T_{0},\ z\in S_{t}(\omega^{-}-2\theta,\omega-2\theta)\subset S_{t-T_{1}}(\omega^{-}-\theta,\omega-\theta).
\end{equation}
On the other hand, Lemma \ref{lem4.1} yields that
$$u(s+t,z)\geq \underline{u}^{t}(s,z)\ \forall t\geq0, s\geq0, z\in\mathbf{R}.$$
Thus for $t\geq T_{0}$
$$u(t,z)\geq \underline{u}^{t-T_{1}}(T_{1},z)\geq 1-2\sigma\ \forall z\in S_{t}(\omega^{-}-2\theta,\omega-2\theta).$$
Therefore, $\liminf\limits_{t\to+\infty}\{\inf\limits_{x\in S_{t}(\omega^{-}-2\theta,\omega-2\theta)}u(t,x)\}=1$.
\end{proof}

\subsection{Examples}
In this subsection, we will give two examples to show that there are many kernels satisfying (K3) and (K5).
\begin{example}\label{eg4.1}
Assume that $K(x,y)\geq\theta_{0}$, when $|x-y|\leq\Delta$. Then (K3) and (K5) hold. Moreover,
$\displaystyle{\lim_{p\rightarrow+\infty}}\frac{\underline{H}(\pm p)}{p}=+\infty.$
\end{example}
\begin{proof}
Let $\eta_{0}=\frac{\Delta}{4}, \delta_{0}=\frac{\Delta}{2}$. Then for any $\eta\geq\eta_{0}$, we have
\begin{equation*}
\begin{split}
\int_{\mathbf{R}}K(x,y)\chi_{\eta+\delta_{0}}(y)\phi(y)dy
&=\int_{B_{\eta+\delta_{0}}(x)}K(x,y)\phi(y)dy\\
&\geq\int_{B_{\eta+\delta_{0}}(x)\cap\{|y|\leq\eta:|x-y|\leq\Delta, |x|\leq\eta+\delta_{0}\}}K(x,y)\phi(y)dy\\
&\geq \theta_{0}\min\limits_{|y|\leq\eta}\phi(y)|B_{\eta+\delta_{0}}(x)\cap\{|y|\leq\eta:|x-y|\leq\Delta, |x|\leq\eta+\delta_{0}\}|\\
&\geq\frac{\Delta}{2}\theta_{0}\min\limits_{|y|\leq\eta}\phi(y).
\end{split}
\end{equation*}
Hence (K3) holds.

Note also that $\Delta$ can be small enough to make sure that $u_{0}(x)>0\ \forall x\in B_{\eta_{0}}(0)$.
For any $t>0, x\in\mathbf{R},$ one can find $n$ large enough such that $|x|<\eta_{0}+n\delta_{0}.$
Then one can easily verify that $K^{n}u_{0}\geq C^{n}\min\limits_{|y|\leq\eta_{0}}u_{0}(y)>0$.
Hence (K5) holds. Moreover, taking $\phi=1$ as a test function, we have
\begin{equation*}
\begin{split}
\displaystyle{\lim_{p\rightarrow +\infty}}\frac{\underline{H}(\pm p)}{p}
&\geq\displaystyle{\lim_{p\rightarrow +\infty}}
\frac{\displaystyle{\lim_{R\rightarrow +\infty}}\inf\limits_{x\in I_{R}}
\{\int_{\mathbf{R}}K(x,y)e^{\pm p(y-x)}dy-a(x)\}}{p}\\
&\geq\displaystyle{\lim_{p\rightarrow +\infty}}
\frac{\displaystyle{\liminf_{R\rightarrow +\infty}}\inf\limits_{x\in I_{R}}
\{\int_{B_{\delta_{0}}(0)}K(x,x-\xi)e^{\mp p\xi}dy-a(x)\}}{p}\\
&\geq\displaystyle{\lim_{p\rightarrow +\infty}}
\frac{\theta_{0}\int_{B_{\delta_{0}}(0)}e^{\mp p\xi}dy-\sup_{x\in\mathbf{R}}a(x)}{p}\\
&\geq\displaystyle{\lim_{p\rightarrow +\infty}}
\frac{\theta_{0}(e^{|p|\delta_{0}}-1)}{p^{2}}=+\infty.
\end{split}
\end{equation*}
\end{proof}

\begin{rem}\label{re4.3}
An observation is that: assume that $a(\cdot)\in C(\mathbf{R})$ is periodic with period $L$ and
$\sum\limits_{l=-\infty}^{\infty}K(x,y+lL)=\sum\limits_{l=-\infty}^{\infty}K(x+L,y+L+lL)$. It follows
form \cite{C10} or \cite{LCW} that there exists an eigenpair $(\lambda_{per},\phi_{per})$ such that
$L_{p}\phi_{per}=\lambda_{per}\phi_{per}$.
Hence $\overline{\lambda_{1}}(p,-\infty)=\underline{{\lambda}_{1}}(p,-\infty)=\lambda_{per}$ by
Corollary \ref{cor2.1},
which yields that $\underline{\omega}=\overline{\omega}$ is exactly the spreading speed provided
(K1), (K2) and (K4) hold.
\end{rem}

\begin{example}\label{eg4.2}
Let $K(x,y)=\sum\limits_{n=1}^{\infty}a_{n}\delta_{q_{n}}(x-y)$,
where $a_{n}$ is a positive sequence, $q_{n}\in\mathbf{R}$
and $\delta$ is the Dirac's delta function on $\mathbf{R}$.
Assume that there exists $q_{n_{1}}>0$ and $q_{n_{2}}<0$
such that $q_{n_{1}}<a_{0}+b_{0}$ and $-q_{n_{2}}<a_{0}+b_{0}$, where $a_{0}=\sup\{a:u_{0}(x)>0\ \forall x\in(-a,0]\}$, $b_{0}=\sup\{b:u_{0}(x)>0\ \forall x\in[0,b)\}$.
Then (K3) and (K5) hold. Moreover, $\displaystyle{\lim_{p\rightarrow+\infty}}\frac{\underline{H}(\pm p)}{p}=+\infty.$
\end{example}
\begin{proof}
Let $\delta_{0}=\min\{q_{n_{1}}, -q_{n_{2}}\}$, $\eta_{0}=\max\{q_{n_{1}}, -q_{n_{2}}\}$.
For any $\eta\geq\eta_{0}$, we consider the following two cases:\\
Case 1: $x\in[0,\eta+\delta_{0}].$ Then
$$K[\chi_{\eta+\delta_{0}}\phi](x)\geq a_{n_{1}}\chi_{\eta+\delta_{0}}(x-q_{n_{1}})\phi(x-q_{n_{1}})
=a_{n_{1}}\phi(x-q_{n_{1}})\geq a_{n_{1}}\min\limits_{|y|\leq\eta}\phi(y)$$
since $-\eta\leq-q_{n_{1}}\leq x-q_{n_{1}}\leq\eta+\delta_{0}-q_{n_{1}}\leq\eta.$\\
Case 2: $x\in[-\eta-\delta_{0},0].$ Then
$$K[\chi_{\eta+\delta_{0}}\phi](x)\geq a_{n_{2}}\chi_{\eta+\delta_{0}}(x-q_{n_{2}})\phi(x-q_{n_{2}})
=a_{n_{2}}\phi(x-q_{n_{2}})\geq a_{n_{2}}\min\limits_{|y|\leq\eta}\phi(y)$$
since $\eta\geq-q_{n_{2}}\geq x-q_{n_{2}}\geq-\eta-\delta_{0}-q_{n_{2}}\geq-\eta.$

Therefore, $K[\chi_{\eta+\delta_{0}}\phi](x)\geq \min\{a_{n_{1}},a_{n_{2}}\}\min\limits_{|y|\leq\eta}\phi(y)$
for any $|x|\leq\eta+\delta_{0}$. Hence (K3) holds. One can easily verify that $Ku_{0}(x)>0\ \forall x\in(q_{n_{2}}-a_{0},q_{n_{1}}+b_{0})$. Then we have
$K^{n}u_{0}(x)>0\ \forall x\in(nq_{n_{2}}-a_{0},nq_{n_{1}}+b_{0})$ by induction. Therefore, (K5) holds.
Moreover, taking $\phi=1$ as a test function, we have
\begin{equation*}
\begin{split}
\displaystyle{\lim_{p\rightarrow +\infty}}\frac{\underline{H}(\pm p)}{p}
&\geq\displaystyle{\lim_{p\rightarrow +\infty}}
\frac{\displaystyle{\lim_{R\rightarrow +\infty}}\inf\limits_{x\in I_{R}}
\{\int_{\mathbf{R}}K(x,y)e^{\pm p(y-x)}dy-a(x)\}}{p}\\
&\geq\displaystyle{\lim_{p\rightarrow +\infty}}
\frac{\displaystyle{\liminf_{R\rightarrow +\infty}}\inf\limits_{x\in I_{R}}
\{\sum_{n=1}^{\infty}a_{n}e^{\mp pq_{n}}-a(x)\}}{p}\\
&\geq\displaystyle{\lim_{p\rightarrow +\infty}}
\frac{a_{n_{1}}e^{\mp pq_{n_{1}}}+a_{n_{2}}e^{\mp pq_{n_{2}}}-\sup_{x\in\mathbf{R}}a(x)}{p}\\
&=+\infty.
\end{split}
\end{equation*}
\end{proof}

\section{Almost periodic coefficients and periodic coefficients}

\subsection{An auxiliary nonlinear equation}
Before going any further, we consider the existence and uniqueness of the bounded solution of the following equation first.
\begin{equation}\label{5.1}
\varepsilon u^{\varepsilon}(x)-\int_{\mathbf{R}}K_{p}(x,y)e^{u^{\varepsilon}(y)-u^{\varepsilon}(x)}dy+a(x)=0,
\end{equation}
where $\varepsilon$ is a parameter and $K_{p}(x,y)=K(x,y)e^{p(y-x)}$. We give the following assumption:\\
(K6) For any fixed $R>0$, $\int_{B_{R}(0)}|K_{p}(x,x-s)-K_{p}(x+\xi,x+\xi-s)|ds\to0$ as $|\xi|\to0$
uniformly w.r.t $x\in\mathbf{R}$.
\begin{thm}\label{thm5.1}
Fix $\varepsilon>0$. Assume that (K1), (K2) and (K6) hold, $a(x)$ is uniformly continuous and $w, v$ are bounded, satisfying
$$\left\{
   \begin{aligned}
   \varepsilon w(x)-\int_{\mathbf{R}}K_{p}(x,y)e^{w(y)-w(x)}dy+a(x)\leq0,\ x\in\mathbf{R},\\
   \varepsilon v(x)-\int_{\mathbf{R}}K_{p}(x,y)e^{v(y)-v(x)}dy+a(x)\geq0,\ x\in\mathbf{R}.
   \end{aligned}
   \right.$$
Then $w(x)\leq v(x)$ on $\mathbf{R}$.
\end{thm}
\begin{proof}
Set $$\Phi(x,y)=w(x)-v(y)-\alpha|x-y|-\mu(|x|+|y|)$$
for $\alpha>0,\ \mu>0$. Then $\Phi$ reaches its maximum at some point, say $(\xi,\eta)$,
over $\mathbf{R}^{2}$. Obviously, $(\xi,\eta)$ depends on $\alpha,\ \mu$.
If $w(x)\leq v(x)$ is not true, then there must exist $x_{0}\in\mathbf{R},\ \delta>0$ such that
$w(x_{0})-v(x_{0})\geq2\delta$. Now for sufficiently small $\mu$,
we have $\Phi(x_{0},x_{0})=w(x_{0})-v(x_{0})-2\mu |x_{0}|\geq\delta$, hence
$0<\delta\leq\Phi(\xi,\eta)\leq\sup\limits_{x}|w(x)|+\sup\limits_{x}|v(x)|$. From this we obtain
$$\delta+\alpha|\xi-\eta|+\mu(|\xi|+|\eta|)\leq w(\xi)-v(\eta)\leq\sup\limits_{x}|w(x)|+\sup\limits_{x}|v(x)|,$$
which yields $|\xi-\eta|\to0$ as $\alpha\to+\infty$ uniformly with respect to $\mu$. Furthermore,
\begin{equation}\label{t5.1.1}
\begin{split}
\varepsilon\delta
&\leq\varepsilon w(\xi)-\varepsilon v(\eta)\\
&\leq \int_{\mathbf{R}}K_{p}(\xi,y)e^{w(y)-w(\xi)}dy-a(\xi)
-\int_{\mathbf{R}}K_{p}(\eta,y)e^{v(y)-v(\eta)}dy+a(\eta)\\
&=\int_{\mathbf{R}}K_{p}(\xi,\xi-s)e^{w(\xi-s)-w(\xi)}ds-\int_{\mathbf{R}}K_{p}(\eta,\eta-s)e^{v(\eta-s)-v(\eta)}ds
+a(\eta)-a(\xi)\\
&\leq\int_{\mathbf{R}}K_{p}(\xi,\xi-s)e^{v(\eta-s)-v(\eta)+\mu( |\eta-s|+|\xi-s|-|\eta|-|\xi|)}ds-\int_{\mathbf{R}}K_{p}(\eta,\eta-s)e^{v(\eta-s)-v(\eta)}ds
+a(\eta)-a(\xi)\\
&\leq\int_{B_{R}(0)}\big(K_{p}(\xi,\xi-s)e^{2\mu |s|}- K_{p}(\eta,\eta-s)\big)e^{v(\eta-s)-v(\eta)}ds
+a(\eta)-a(\xi)+\frac{\varepsilon\delta}{2}\\
&\leq\int_{B_{R}(0)}\big|(K_{p}(\xi,\xi-s)- K_{p}(\eta,\eta-s))\big|e^{2\mu |s|}
e^{v(\eta-s)-v(\eta)}ds\\
&\quad+\int_{B_{R}(0)}K_{p}(\eta,\eta-s)(e^{2\mu|s|}-1)
e^{v(\eta-s)-v(\eta)}ds+a(\eta)-a(\xi)+\frac{\varepsilon\delta}{2}.
\end{split}
\end{equation}
The third inequality is valid because $\Phi$ reaches its maximum at $(\xi,\eta)$ over $\mathbf{R}^{2}$.
The last inequality is because of (K2) and (K6).
Taking $\mu\to0$ and $\alpha\to+\infty$ (after passing a subsequence) in \eqref{t5.1.1}, we have
$\frac{\varepsilon\delta}{2}\leq0$, which is a contradiction!
\end{proof}
Denote
$$\text{USC}(\mathbf{R})=\{\text{upper semicontinuous functions}\ v:\mathbf{R}\to \mathbf{R}\},$$
$$\text{LSC}(\mathbf{R})=\{\text{lower semicontinuous functions}\ v:\mathbf{R}\to \mathbf{R}\}.$$
\begin{thm}\label{thm-exist-uniq}
Assume that $K:\text{USC}(\mathbf{R})\to\text{USC}(\mathbf{R})$ and (K1), (K2) and (K6) hold. For any fixed $\varepsilon>0$, there is a unique solution $u^{\varepsilon}\in C(\mathbf{R})\cap L^{\infty}(\mathbf{R})$ of
equation \eqref{5.1} such that
\begin{equation}\label{subsuper}
\inf\limits_{x}\frac{\int_{\mathbf{R}}K(x,y)e^{p(y-x)}dy-a(x)}{\varepsilon}\leq
u^{\varepsilon}(x)\leq
\sup\limits_{x}\frac{\int_{\mathbf{R}}K(x,y)e^{p(y-x)}dy-a(x)}{\varepsilon}.
\end{equation}
\end{thm}
Denote $\underline{c}=\inf\limits_{x}\{\int_{\mathbf{R}}K(x,y)e^{p(y-x)}dy-a(x)\}$ and
$\overline{c}=\sup\limits_{x}\{\int_{\mathbf{R}}K(x,y)e^{p(y-x)}dy-a(x)\}$.
Then the solution of \eqref{5.1} in $L^{\infty}(\mathbf{R})$ must satisfies \eqref{subsuper} by
Theorem \ref{thm5.1}.
Hence we only need to show the existence.
Let
$$A=\{v\in \text{USC}(\mathbf{R}): \frac{\underline{c}}{\varepsilon}\leq v(x)\leq\frac{\overline{c}}{\varepsilon}+1, v\ \text{is a subsolution}\}.$$
Here, we say $v$ is a subsolution (supersolution) if
$\varepsilon w(x)-\int_{\mathbf{R}}K_{p}(x,y)e^{w(y)-w(x)}dy+a(x)\leq(\geq)0$.
We prove the existence in several lemmas by using Perron's method.

\begin{lemma}\label{lem-e-u1}
Let $w(x):=\displaystyle{\sup_{v\in A}}v(x).$ Then $w\in A.$
\end{lemma}
\begin{proof}
It is obvious that $\frac{\underline{c}}{\varepsilon}\leq w(x)\leq\frac{\overline{c}}{\varepsilon}+1$
for any $x\in\mathbf{R}$. In fact, $w(x)\leq\frac{\overline{c}}{\varepsilon}$
since $\frac{\overline{c}}{\varepsilon}$ is a supersolution.
Now for any fixed $x_{0}\in\mathbf{R}$, any $\delta>0$,
there exist $v^{\delta}\in A$ such that $v^{\delta}(x_{0})\geq w(x_{0})-\delta$. Therefore,
$$\displaystyle{\liminf_{x\rightarrow x_{0}}}w(x)\geq
\displaystyle{\liminf_{x\rightarrow x_{0}}}v^{\delta}(x)\geq
v^{\delta}(x_{0})\geq w(x_{0})-\delta.$$
Since $\delta$ is arbitrarily, we have $\displaystyle{\liminf_{x\rightarrow x_{0}}}w(x)\geq w(x_{0})$.
Hence $w\in\text{LSC}(\mathbf{R})$. Next, we prove that $w$ is a subsolution. If not, then there exist
$x_{0}$ and $\eta>0$ such that
$\varepsilon w(x_{0})-\int_{\mathbf{R}}K_{p}(x_{0},y)e^{w(y)-w(x_{0})}dy+a(x_{0})\geq\eta.$ Hence
\begin{equation*}
\begin{split}
&\ \ \ \ \varepsilon v^{\delta}(x_{0})
 -\int_{\mathbf{R}}K_{p}(x_{0},y)e^{v^{\delta}(y)-v^{\delta}(x_{0})}dy+a(x_{0})\\
&\geq\varepsilon w(x_{0})-\varepsilon\delta
 -\int_{\mathbf{R}}K_{p}(x_{0},y)e^{w(y)-v^{\delta}(x_{0})}dy+a(x_{0})\\
&\geq\varepsilon w(x_{0})-\varepsilon\delta
 -\int_{\mathbf{R}}K_{p}(x_{0},y)e^{w(y)-w(x_{0})+\delta}dy+a(x_{0})\\
&\geq\eta-\varepsilon\delta
 -\int_{\mathbf{R}}K_{p}(x_{0},y)e^{w(y)-w(x_{0})}(e^{\delta}-1)dy\\
&\geq\eta-\varepsilon\delta
 -\int_{\mathbf{R}}K_{p}(x_{0},y)dy\exp(\frac{\overline{c}-\underline{c}}{\varepsilon})(e^{\delta}-1)\\
&>0\\
\end{split}
\end{equation*}
for $\delta$ small enough, which is contradicts $v^{\delta}\in A$. From all above, we know that $w\in A$.
\end{proof}

\begin{lemma}\label{lem-e-u2}
$w$ is a supersolution. Moreover,
$\varepsilon w(x)-\int_{\mathbf{R}}K_{p}(x,y)e^{w(y)-w(x)}dy+a(x)=0.$
\end{lemma}
\begin{proof}
It is sufficient to show $\varepsilon w(x)-\int_{\mathbf{R}}K_{p}(x,y)e^{w(y)-w(x)}dy+a(x)\geq0.$
If not, then there exist $x_{0}$ and $\eta>0$ such that
$(\varepsilon w(x_{0})+a(x_{0}))e^{w(x_{0})}-\int_{\mathbf{R}}K_{p}(x_{0},y)e^{w(y)}dy\leq-\eta.$
First, we have
$$\liminf \limits_{x\to x_{0}}\int_{\mathbf{R}}K_{p}(x,y)e^{w(y)}dy\geq
\int_{\mathbf{R}}\liminf \limits_{x\to x_{0}}K_{p}(x,y)e^{w(y)}dy
=\int_{\mathbf{R}}K_{p}(x_{0},y)e^{w(y)}dy$$
by Fatou's Lemma, 
i.e., $\int_{\mathbf{R}}K_{p}(x,y)e^{w(y)}dy$ is a lower
semicontinuous function. Then for any $\delta>0$, there exists $\sigma_{1}>0$ such that
\begin{equation}\label{Klsc}
\int_{\mathbf{R}}K_{p}(x,y)e^{w(y)}dy\geq\int_{\mathbf{R}}K_{p}(x_{0},y)e^{w(y)}dy-\delta,\ \forall x\in B_{\sigma_{1}}(x_{0}).
\end{equation}
Claim: $\liminf \limits_{x\to x_{0}}w(x)=w(x_{0}).$\\
Proof of Claim: If $\liminf \limits_{x\to x_{0}}w(x)>w(x_{0})$, then we can prove that $\tilde{w}$
is a subsolution by a similar computation to Lemma \ref{lem-e-u1}, where
$$\tilde{w}(x)=\left\{
   \begin{aligned}
 w(x)& \ \ x\neq x_{0},\\
 w(x)+\delta& \ \ x=x_{0},\\
   \end{aligned}
   \right.$$
with $\delta\in(0,\liminf \limits_{x\to x_{0}}w(x)-w(x_{0}))$ sufficiently small. Hence $\tilde{w}\in A$.
In particular, $\tilde{w}(x_{0})\leq w(x_{0})$, which is contradicts $\tilde{w}(x_{0})=w(x_{0})+\delta$.
Thus $\liminf \limits_{x\to x_{0}}w(x)=w(x_{0}).$

Note that $\limsup \limits_{x\to x_{0}^{\pm}}w(x)\geq\liminf \limits_{x\to x_{0}^{\pm}}w(x)\geq\liminf \limits_{x\to x_{0}}w(x)=w(x_{0}).$ Then there are two cases we need consider.\\
Case 1: $\limsup \limits_{x\to x_{0}^{+}}w(x)=w(x_{0})$ or
$\limsup \limits_{x\to x_{0}^{-}}w(x)=w(x_{0})$.\\
We will prove the case where $\limsup \limits_{x\to x_{0}^{-}}w(x)=w(x_{0})$. One can prove similarly when $\limsup \limits_{x\to x_{0}^{+}}w(x)=w(x_{0})$.
For any $\delta>0$, there exists $\sigma_{2}>0$ such that $|w(x)-w(x_{0})|\leq\delta$ for any $x\in(x_{0}-\sigma_{2},x_{0}).$
Let $0<\sigma_{0}<\min\{\sigma_{1},\sigma_{2}\}$ and $\delta\leq\frac{\eta}{4}$ will be chosen later.
Then for any $x\in(x_{0}-\sigma_{0},x_{0}),$ we have
\begin{equation}\label{Klsc1}
\begin{split}
&\ \ \ \ (\varepsilon w(x)+a(x))e^{w(x)}-\int_{\mathbf{R}}K_{p}(x,y)e^{w(y)}dy\\
&\leq(\varepsilon w(x_{0})+a(x_{0}))e^{w(x_{0})}-\int_{\mathbf{R}}K_{p}(x_{0},y)e^{w(y)}dy+\delta\\
&\quad-(\varepsilon w(x_{0})+a(x_{0}))e^{w(x_{0})}+(\varepsilon w(x)+a(x))e^{w(x)}\\
&\leq-\eta+\delta-(\varepsilon w(x_{0})+a(x_{0}))e^{w(x_{0})}+(\varepsilon w(x)+a(x))e^{w(x)}\\
&\leq-\frac{3\eta}{4}+\big((\varepsilon w(x)+a(x))e^{w(x)}-(\varepsilon w(x_{0})+a(x_{0}))e^{w(x_{0})}\big)\\
&\leq-\frac{\eta}{2}
\end{split}
\end{equation}
for $\sigma_{0}$ small enough by the continuity of $w$ and $a$.
Consider $\tilde{w}=w+\rho$, where $\rho(x)\geq0$ is smooth with $\text{supp}\rho\subset(x_{0}-\sigma_{0},x_{0})$.
Denote $\rho_{0}=\sup \rho(x)$. Then we can prove that $\tilde(w)$ is a subsolution. In fact,
if $x\notin(x_{0}-\sigma_{0},x_{0})$, then
\begin{equation*}
\begin{split}
&\ \ \ \ (\varepsilon \tilde{w}(x)+a(x))e^{\tilde{w}(x)}-\int_{\mathbf{R}}K_{p}(x,y)e^{\tilde{w}(y)}dy\\
&=(\varepsilon {w}(x)+a(x))e^{{w}(x)}-\int_{\mathbf{R}}K_{p}(x,y)e^{\tilde{w}(y)}dy\\
&\leq(\varepsilon {w}(x)+a(x))e^{{w}(x)}-\int_{\mathbf{R}}K_{p}(x,y)e^{{w}(y)}dy\\
&\leq0.
\end{split}
\end{equation*}
If $x\in(x_{0}-\sigma_{0},x_{0})$, then using \eqref{Klsc1}, we have
\begin{equation*}
\begin{split}
&\ \ \ \ (\varepsilon \tilde{w}(x)+a(x))e^{\tilde{w}(x)}-\int_{\mathbf{R}}K_{p}(x,y)e^{\tilde{w}(y)}dy\\
&\leq(\varepsilon \tilde{w}(x)+a(x))e^{\tilde{w}(x)}-\int_{\mathbf{R}}K_{p}(x,y)e^{{w}(y)}dy\\
&=(\varepsilon {w}(x)+a(x))e^{{w}(x)}-\int_{\mathbf{R}}K_{p}(x,y)e^{{w}(y)}dy\\
&\quad+(\varepsilon \tilde{w}(x)+a(x))e^{\tilde{w}(x)}-(\varepsilon {w}(x)+a(x))e^{{w}(x)}\\
&\leq-\frac{\eta}{2}+(\varepsilon \tilde{w}(x)+a(x))e^{\tilde{w}(x)}-(\varepsilon {w}(x)+a(x))e^{{w}(x)}\\
&\leq0
\end{split}
\end{equation*}
for $\rho_{0}$ small enough by the boundedness of $w$ and $a$.
Therefore, we obtain $\tilde{w}\in A$, which is contradicts $\tilde{w}\gvertneqq w$.\\
Case 2: $\limsup \limits_{x\to x_{0}^{\pm}}w(x)>w(x_{0})$.\\
There exist $\theta>0$, an increasing sequence $\{x_{n}\}_{n=1}^{\infty}$ with $\lim\limits_{n\to\infty}x_{n}=x_{0}$,
and a decreasing sequence $\{y_{n}\}_{n=1}^{\infty}$ with $\lim\limits_{n\to\infty}y_{n}=x_{0}$ such that
$w(x_{n})>w(x_{0})+\theta$ and $w(y_{n})>w(x_{0})+\theta$.
Moreover, for any fixed $n\in\mathbf{N}$, there always exist neighbourhoods $U(x_{n})$ and $U(y_{n})$
such that $w(x)\geq w(x_{n})-\frac{\theta}{2}>w(x_{0})+\frac{\theta}{2}$  for any $x\in U(x_{n})$ and
$w(x)\geq w(y_{n})-\frac{\theta}{2}>w(x_{0})+\frac{\theta}{2}$ for any $x\in U(y_{n})$. Therefore,
we can find a function $\chi_{n}$ satisfies:\\
(i). $\chi_{n}$ is continuous on $\mathbf{R}\setminus\{x_{n},y_{n}\},$\\
(ii). $$\chi_{n}(x)\left\{
   \begin{aligned}
 &=w(x)+\delta,\ &\ x\in(x_{n},y_{n}),\\
 &\leq w(x),\ &\ x\notin(x_{n},y_{n}),\\
   \end{aligned}
   \right.$$
where $\delta\in(0,\max\{\frac{\theta}{2}, \frac{\eta}{2}, 1\})$ will be chosen later.

Let ${w}_{n}(x)=\max\{w(x), \chi_{n}(x)\}$. then ${w}_{n}$ is lower semicontinuous and
$w_{n}(x_{0})=w(x_{0})+\delta>w(x_{0})$. For $\sigma_{1}$ given in \eqref{Klsc}, there exists
$N\in\mathbf{N}$ such that $(x_{n},y_{n})\subset B_{\sigma_{1}}(x_{0})$ for any $n\geq N$.
That is to say, for any $\delta>0$, there exists $N\in\mathbf{N}$ such that \eqref{Klsc}
is still valid as long as $x\in (x_{n},y_{n})$ with $n\geq N$.

If $x\in\{x: \chi_{n}\leq w(x)\}$, then $w_{n}(x)=w(x)$ and
\begin{equation}\label{Klsc2}
\begin{split}
&\ \ \ \ (\varepsilon w_{n}(x)+a(x))e^{w_{n}(x)}-\int_{\mathbf{R}}K_{p}(x,y)e^{w_{n}(y)}dy\\
&=(\varepsilon w(x)+a(x))e^{w(x)}-\int_{\mathbf{R}}K_{p}(x,y)e^{w_{n}(y)}dy\\
&\leq(\varepsilon w(x)+a(x))e^{w(x)}-\int_{\mathbf{R}}K_{p}(x,y)e^{w(y)}dy\\
&\leq0.
\end{split}
\end{equation}
If $x\in\{x: \chi_{n}>w(x)\}=\{x: w(x)<w(x_{0})+\delta\}\cap(x_{n},y_{n})\subset(x_{n},y_{n})$,
then $w_{n}(x)=w(x_{0})+\delta$. Using \eqref{Klsc}, we have
\begin{equation}\label{Klsc3}
\begin{split}
&\ \ \ \ (\varepsilon w_{n}(x)+a(x))e^{w_{n}(x)}-\int_{\mathbf{R}}K_{p}(x,y)e^{w_{n}(y)}dy\\
&\leq(\varepsilon w(x_{0})+a(x_{0}))e^{w(x_{0})}-\int_{\mathbf{R}}K_{p}(x,y)e^{w(y)}dy\\
&\quad+(\varepsilon w_{n}(x)+a(x))e^{w_{n}(x)}-(\varepsilon {w}(x_{0})+a(x_{0}))e^{{w}(x_{0})}\\
&\leq(\varepsilon w(x_{0})+a(x_{0}))e^{w(x_{0})}-\int_{\mathbf{R}}K_{p}(x_{0},y)e^{w(y)}dy+\delta\\
&\quad+(\varepsilon w_{n}(x)+a(x))e^{w_{n}(x)}-(\varepsilon {w}(x_{0})+a(x_{0}))e^{{w}(x_{0})}\\
&\leq-\frac{\eta}{2}+(\varepsilon w_{n}(x)+a(x))e^{w_{n}(x)}-(\varepsilon {w}(x_{0})+a(x_{0}))e^{{w}(x_{0})}\\
&\leq0
\end{split}
\end{equation}
for $\delta$ small and $N$ large enough by the continuity of $a$.
It follows from \eqref{Klsc2} and \eqref{Klsc3} that $w_{n}$ is a subsolution when $n$ is large enough.
Hence $w_{n}\in A$, which is contradicts ${w}_{n}\gvertneqq w$.
Both Case 1 and Case 2 can not occur. Thus the proof is complete.
\end{proof}

Let $\overline{A}=\{v\in \text{LSC}(\mathbf{R}): \frac{\underline{c}}{\varepsilon}-1\leq v(x)\leq\frac{\overline{c}}{\varepsilon}, v\ \text{is a supersolution}\}$
and $\overline{w}(x):=\displaystyle{\inf_{v\in A}}v(x).$ Then $\overline{w}\in\overline{A}$ and
$$\varepsilon \overline{w}(x)-\int_{\mathbf{R}}K_{p}(x,y)e^{\overline{w}(y)-\overline{w}(x)}dy+a(x)=0.$$
One can prove it by almost the same argument as before. We only point out that we need the assumption
$K:\text{USC}(\mathbf{R})\to\text{USC}(\mathbf{R})$ because we can't use Fatou's Lemma this time.
By Theorem \ref{thm5.1}, we obtain that the solution $w(x)=\overline(w)(x)$ is continuous.

\subsection{Almost periodic coefficients}
In this subsection,  we always assume that\\
(K3)$^{\prime}$ $K(x,y)\geq\theta_{0}$ for $|x-y|\leq\Delta$, and $K(x,y)\leq k_{0}\ \forall(x,y)\in\mathbf{R}^{2}$, and\\
(K7) $\exists \ \varepsilon_{0}>0$ and $r(\varepsilon)$ s.t.
$$\alpha_{p}:=\sup\limits_{x\in\mathbf{R},\varepsilon\in[0,\varepsilon_{0}]}
\frac{\int_{B_{r(\varepsilon)}^{c}(x)}K_{p}(x,y)dy}
{\int_{\mathbf{R}}K_{p}(x,y)dy}\exp(\frac{\overline{c}-\underline{c}}{\varepsilon})<1$$
and $\varepsilon r(\varepsilon)\to0$ as $\varepsilon\to0^{+}$.\\
One can easily verify that (K7) holds if $K$ is compact support. By compact support we mean that:\\
(K7)$^{\prime}$ $\exists r_{0}(>\Delta_{0})$ s.t. $K(x,y)=0$ if $|x-y|\geq r_{0}$. (This yields (K2).)

Under some assumptions (c.f. Theorem \ref{thm5.3}), we will prove that
$\underline{\omega}=\overline{\omega}$
when the media is almost periodic, here almost periodic media means that:\\
(K8) For any sequence $\{x_{n}\}$ there exists a subsequence still denoted by $x_{n}$ s.t.
\begin{equation}\label{5.2.1}
\int_{B_{r_{0}}(0)}\big|K_{p}(x+x_{n},x+x_{n}-s)-K_{p}(x+x_{m},x+x_{m}-s)\big|ds\to0\ \text{as}\ n,m\to\infty
\end{equation}
uniformly w.r.t $x\in\mathbf{R}$, and\\
(K9) $a(x)=b(x)-f_{s}^{\prime}(x,0)$ is almost periodic, i.e., for any sequence $x_{n}$,
there exists a subsequence $x_{n_{k}}$ such that $a(\cdot+x_{n_{k}})$ converges in
$C(\mathbf{R})$.\\

We need the following Harnack type inequality
\begin{thm}\label{thm5.2}
Assume that (K1), (K2) and (K6) hold, $u^{\varepsilon}$ satisfies \eqref{5.1}.
For any $x,y\in\mathbf{R}$ with $|x-y|<r$, there exist constants $C_{0}, C_{1}$ depending on
$r, K, p,$ and $f_{s}^{\prime}$, but independent of $\varepsilon, x,$ and $y$ such that
$$\phi^{\varepsilon}(y)\geq C_{0}C_{1}^{O(r(\varepsilon))}(a(x)+\varepsilon u^{\varepsilon}(x))\phi^{\varepsilon}(x),$$
where $\phi^{\varepsilon}=e^{u^{\varepsilon}}$.
\end{thm}

\begin{lemma}\label{lem5.1}
Assume that (K1), (K2) and (K6) hold (except that $K$ is bounded),
$u^{\varepsilon}$ satisfies \eqref{5.1}. For any $y\in\mathbf{R}$, $0<\delta\leq\Delta$,
there exists a constant $C_{0}$ depending on $K, p, \delta$ and $f_{s}^{\prime}$, but independent of $\varepsilon, y$ such that
$$\phi(y)\geq C_{0}\int_{B_{\delta}(y)}\phi(s)ds.$$
\end{lemma}
\begin{proof}
See\cite[Lemma 2.1]{C}.
\end{proof}

\begin{lemma}\label{lem5.2}
Assume that (K1), (K2) and (K6) hold (except that $K$ is bounded),
$u^{\varepsilon}$ satisfies \eqref{5.1}. For any $y\in\mathbf{R}$ and $r\geq\Delta$,
there exist $d(\geq\frac{\Delta}{6})$ depending on $\Delta$ and $C_{1}$ only depending on $K, p, \Delta$
and $f_{s}^{\prime}$, but independent of $r$ $\varepsilon, y$ such that
$$\int_{B_{r}(y)}\phi(s)ds\geq C_{1}\int_{B_{r+d}(y)}\phi(s)ds.$$
\end{lemma}
\begin{proof}
See \cite[Lemma 2.5]{C}.
\end{proof}

\begin{proof}[Proof of Theorem \ref{thm5.2}]
First from \eqref{5.1}, we have
\begin{equation*}
\begin{split}
(a(x)+\varepsilon u^{\varepsilon}(x))\phi(x)
&=\int_{\mathbf{R}}K_{p}(x,y)\phi(y)dy\\
&=\int_{B_{r(\varepsilon)}(x)}K_{p}(x,y)\phi(y)dy+\int_{B_{r(\varepsilon)}^{c}(x)}K_{p}(x,y)\phi(y)dy\\
&=\int_{B_{r(\varepsilon)}(x)}K_{p}(x,y)\phi(y)dy+\frac{\int_{B_{r(\varepsilon)}^{c}(x)}K_{p}(x,y)\phi(y)dy}
{\int_{\mathbf{R}}K_{p}(x,y)\phi(y)dy}\int_{\mathbf{R}}K_{p}(x,y)\phi(y)dy\\
&\leq\int_{B_{r(\varepsilon)}(x)}K_{p}(x,y)\phi(y)dy+\alpha_{p}\int_{\mathbf{R}}K_{p}(x,y)\phi(y)dy\\
&=\int_{B_{r(\varepsilon)}(x)}K_{p}(x,y)\phi(y)dy+\alpha_{p}(a(x)+\varepsilon u^{\varepsilon}(x))\phi(x).\\
\end{split}
\end{equation*}
Hence,
\begin{equation}\label{l5.2.1}
(a(x)+\varepsilon u^{\varepsilon}(x))\phi(x)
\leq\frac{\int_{B_{r(\varepsilon)}(x)}K_{p}(x,y)\phi(y)dy}{1-\alpha_{p}}
\leq\frac{k_{0}}{1-\alpha_{p}}e^{O(r(\varepsilon))}\int_{B_{r(\varepsilon)}(x)}\phi(y)dy.
\end{equation}
On the other hand, using Lemmas \ref{lem5.1} and \ref{lem5.2}, we have
\begin{equation*}
\begin{split}
\phi(y)\geq C_{0}\int_{B_{\Delta}(y)}\phi(s)ds
\geq C_{0}C_{1}\int_{B_{\Delta+d}(y)}\phi(s)ds\\
\geq C_{0}C_{1}^{2}\int_{B_{\Delta+2d}(y)}\phi(s)ds
\geq C_{0}C_{1}^{O(r(\varepsilon))}\int_{B_{r(\varepsilon)}(x)}\phi(s)ds.\\
\end{split}
\end{equation*}
Combining this with \eqref{l5.2.1}, we have
$$\phi(y)\geq\frac{(1-\alpha_{p})C_{0}}{k_{0}}(\frac{C_{1}}{e})^{O(r(\varepsilon))}(a(x)+\varepsilon u^{\varepsilon}(x))\phi(x).$$
\end{proof}

\begin{rem}\label{re5.1}
If $K(6)^{\prime}$ holds, then the conclusion of Theorem \ref{thm5.2} will be
$$\phi^{\varepsilon}(y)\geq C(a(x)+\varepsilon u^{\varepsilon}(x))\phi^{\varepsilon}(x),$$
where the constant $C$ depends on $r, K, p,$ and
$f_{s}^{\prime}$, but independent of $\varepsilon, x,$ and $y$.
\end{rem}

A useful corollary of Theorem \ref{thm5.2} is that

\begin{cor}\label{cor5.1}
Under the assumptions of Theorem \ref{thm5.2}. We further assume that $a(x)$ is locally Lipschitz continuous.
Then for any fixed $x_{0}\in\mathbf{R}, R>0$, we have
$\limsup\limits_{\varepsilon\to0^{+}}\varepsilon(u^{\varepsilon}(x_{0})-u^{\varepsilon}(\tilde{y}))\leq0$
uniformly with respect to $\tilde{y}\in B_{R}(0)$. In particular,
$\lim\limits_{\varepsilon\to0^{+}}\varepsilon(u^{\varepsilon}(x_{0})-u^{\varepsilon}({y}))=0$
for any fixed $x,y\in\mathbf{R}.$
\end{cor}
\begin{proof}
Claim: For any fixed $x\in\mathbf{R}$, $\liminf\limits_{\varepsilon\to0}\varepsilon\ln(a(x)+\varepsilon u^{\varepsilon}(x))
\geq0$.\\
Proof of claim: There are two cases we need consider:\\
Case 1: $(a(x)+\varepsilon u^{\varepsilon}(x))$ reaches its minimum at some point, say $x_{0}$. Then
$(a(x)+\varepsilon u^{\varepsilon}(x))\geq(a(x_{0})+\varepsilon u^{\varepsilon}(x_{0}))$, hence
\begin{equation*}
\begin{split}
a(x_{0})+\varepsilon u^{\varepsilon}(x_{0})
&=\int_{\mathbf{R}}K_{p}(x_{0},y)e^{u^{\varepsilon}(y)-u^{\varepsilon}(x_{0})}dy\\
&\geq \int_{\mathbf{R}}K_{p}(x_{0},y)\exp(\frac{a(y)-a(x_{0})}{\varepsilon})dy\\
&\geq \int_{B_{\varepsilon}(x_{0})}K_{p}(x_{0},y)\exp(\frac{a(y)-a(x_{0})}{y-x_{0}}\frac{y-x_{0}}{\varepsilon})dy\\
&\geq \int_{B_{\varepsilon}(x_{0})}K_{p}(x_{0},y)\exp(-L(x_{0})\frac{|y-x_{0}|}{\varepsilon})dy\\
&\geq \int_{B_{\varepsilon}(x_{0})}K_{p}(x_{0},y)\exp(-L(x_{0}))dy\\
&\geq2\varepsilon\theta_{0}e^{-|p|\varepsilon}e^{-L(x_{0})}.
\end{split}
\end{equation*}
Hence, $\liminf\limits_{\varepsilon\to0}\varepsilon\ln(a(x)+\varepsilon u^{\varepsilon}(x))\geq
\liminf\limits_{\varepsilon\to0}\varepsilon\ln2\varepsilon\theta_{0}e^{-|p|\varepsilon}e^{-L(x_{0})}=0.$\\
Case 2: The minimum can not be reached. Then there must be $[x_{n},y_{n}]\subset\mathbf{R}$ with
$x_{n}\to-\infty$ and $y_{n}\to+\infty$ such that either
$a(x_{n})+\varepsilon u^{\varepsilon}(x_{n})=
\min\limits_{x\in[x_{n},y_{n}]}a(x)+\varepsilon u^{\varepsilon}(x)\ \forall n$
or
$a(y_{n})+\varepsilon u^{\varepsilon}(y_{n})=
\min\limits_{x\in[x_{n},y_{n}]}a(x)+\varepsilon u^{\varepsilon}(x)\ \forall n.$

If
$a(x_{n})+\varepsilon u^{\varepsilon}(x_{n})=
\min\limits_{x\in[x_{n},y_{n}]}a(x)+\varepsilon u^{\varepsilon}(x)\ \forall n$, then
\begin{equation*}
\begin{split}
a(x_{n})+\varepsilon u^{\varepsilon}(x_{n})
&=\int_{\mathbf{R}}K_{p}(x_{n},y)e^{u^{\varepsilon}(y)-u^{\varepsilon}(x_{n})}dy\\
&\geq \int_{x_{n}}^{x_{n}+\varepsilon}K_{p}(x_{n},y)\exp(\frac{a(y)-a(x_{n})}{\varepsilon})dy\\
&\geq \int_{x_{n}}^{x_{n}+\varepsilon}K_{p}(x_{n},y)\exp(-L(x_{n})\frac{|y-x_{n}|}{\varepsilon})dy\\
&\geq\varepsilon\theta_{0}e^{-|p|\varepsilon}e^{-L(x_{n})}.
\end{split}
\end{equation*}
Note that $x\in[x_{n},y_{n}]$ for a sufficiently large $n$, hence $\liminf\limits_{\varepsilon\to0}\varepsilon\ln(a(x)+\varepsilon u^{\varepsilon}(x))\geq
\liminf\limits_{\varepsilon\to0}\varepsilon\ln2\varepsilon\theta_{0}e^{-|p|\varepsilon}e^{-L(x_{n})}=0.$
One can obtain a similar conclusion when $a(y_{n})+\varepsilon u^{\varepsilon}(y_{n})=
\min\limits_{x\in[x_{n},y_{n}]}a(x)+\varepsilon u^{\varepsilon}(x)\ \forall n.$ Thus the claim is proved .

By Theorem \ref{thm5.2},
\begin{equation}\label{t5.2.1}
e^{u^{\varepsilon}(\tilde{y})-u^{\varepsilon}(x_{0})}\geq C_{0}C_{1}^{O(r(\varepsilon))}(a(x_{0})+\varepsilon u^{\varepsilon}(x_{0}))\ \forall x,\tilde{y}\in B_{R}(0).
\end{equation}
Therefore
\begin{equation*}
\begin{split}
\limsup\limits_{\varepsilon\to0^{+}}\varepsilon(u^{\varepsilon}(x_{0})-u^{\varepsilon}(\tilde{y}))
&\leq\limsup\limits_{\varepsilon\to0^{+}}\big(-\varepsilon\ln C_{0}-\varepsilon O(r(\varepsilon))\ln C_{1}-\varepsilon\ln(a(x_{0})+\varepsilon u^{\varepsilon}(x_{0}))\big)\\
&\leq-\liminf\limits_{\varepsilon\to0^{+}}\varepsilon\ln(a(x_{0})+\varepsilon u^{\varepsilon}(x_{0}))
\leq0
\end{split}
\end{equation*}
uniformly with respect to $\tilde{y}\in B_{R}(0)$.
\end{proof}

For convenience, we set two situations:
$$(S1): (K1), (K3)^{\prime}, (K4), (K6), (K7)^{\prime}, (K8), (K9)\ \text{hold};$$
$$(S2): (K1), (K3)^{\prime}, (K4), (K7), (K9)\ \text{hold}.$$

\begin{thm}\label{thm5.3}
Assume that $a(x)$ is locally Lipschitz continuous,
$u^{\varepsilon}\in C(\mathbf{R})\cap L^{\infty}(\mathbf{R})$
be the solution of \eqref{5.1}. If we further assume one of the following conditions:\\
(i) Under situation (S1) and
\begin{equation}\label{t5.3.0}
\kappa:=\inf\limits_{p}\inf\limits_{x}\int_{\mathbf{R}}K(x,y)e^{p(y-x)}dy-\sup\limits_{x}a(x)+\inf\limits_{x}a(x)>0;
\end{equation}
(ii) Under situation (S2) and $K(x,y)=K(x-y)$.\\
Then $\varepsilon u^{\varepsilon}(x)$ converges to some constant as $\varepsilon\to0$
uniformly with respect to $x\in\mathbf{R}$.
\end{thm}
\begin{proof}
Let $\hat{u}^{\varepsilon}(x):=u^{\varepsilon}(x)-u^{\varepsilon}(0)$. Then $\hat{u}^{\varepsilon}(x)$
satisfies
$$\varepsilon \hat{u}^{\varepsilon}(x)-
\int_{\mathbf{R}}K_{p}(x,y)e^{\hat{u}^{\varepsilon}(y)-\hat{u}^{\varepsilon}(x)}dy+a(x)+
\varepsilon u^{\varepsilon}(0)=0$$
Claim: $\varepsilon\hat{u}^{\varepsilon}(x)\to0$ as $\varepsilon\to0$ uniformly with respect to $x$.\\
Proof of claim:
Assume by contradiction that $\exists\ \varepsilon_{n}\to0,\ x_{n},\ \theta>0$ such that
$|\varepsilon_{n}\hat{u}^{\varepsilon_{n}}(x_{n})|\geq2\theta$. We will prove the case where
$\varepsilon_{n}\hat{u}^{\varepsilon_{n}}(x_{n})\geq2\theta$, and one can prove the case where
$\varepsilon_{n}\hat{u}^{\varepsilon_{n}}(x_{n})\leq-2\theta$ similarly.
Without loss of generality, we may assume that $a(x+x_{n})$ converges uniformly as $n\to+\infty$
since $a(x)$ is almost periodic and that \eqref{5.2.1} holds in $(S1)$.

Set $u_{n}^{\varepsilon}(x):=\hat{u}^{\varepsilon}(x+x_{n})$.
Then $u_{n}^{\varepsilon}(x),\ u_{m}^{\varepsilon}(x)$ satisfy
\begin{equation}\label{t5.3.1}
\varepsilon u_{n}^{\varepsilon}(x)-
\int_{\mathbf{R}}K_{p}(x+x_{n},y+x_{n})e^{u_{n}^{\varepsilon}(y)-u_{n}^{\varepsilon}(x)}dy
+a(x+x_{n})+\varepsilon u^{\varepsilon}(0)=0,\ \text{and}
\end{equation}
\begin{equation}\label{t5.3.2}
\varepsilon u_{m}^{\varepsilon}(x)-
\int_{\mathbf{R}}K_{p}(x+x_{m},y+x_{m})e^{u_{m}^{\varepsilon}(y)-u_{m}^{\varepsilon}(x)}dy+a(x+x_{m})+
\varepsilon u^{\varepsilon}(0)=0.
\end{equation}
Set $w(x):=u_{n}^{\varepsilon}(x)-\frac{\eta_{m,n}}{\varepsilon}$, where
$$\eta_{m,n}:=\|a(\cdot+x_{n})-a(\cdot+x_{m})\|_{\infty}+
\|\int_{\mathbf{R}}\big( K_{p}(\cdot+x_{n},y+x_{n})-K_{p}(\cdot+x_{m},y+x_{m})\big)e^{u_{n}^{\varepsilon}(y)-u_{n}^{\varepsilon}(x)}dy\|_{\infty}.$$
Then $w(x)$ satisfies
$$\varepsilon w(x)-
\int_{\mathbf{R}}K_{p}(x+x_{m},y+x_{m})e^{w(y)-w(x)}dy+a(x+x_{m})+
\varepsilon u^{\varepsilon}(0)\leq0$$
since $w(y)-w(x)=u_{n}^{\varepsilon}(y)-u_{n}^{\varepsilon}(x)$.
Hence by Theorem \ref{thm5.1}, we have $w(x)\leq u_{m}^{\varepsilon}(x)$ for any $x\in\mathbf{R}$, i.e.,
$$\varepsilon\hat{u}^{\varepsilon}(x+x_{n})\leq\varepsilon\hat{u}^{\varepsilon}(x+x_{m})+\eta_{m,n}.$$
Setting $\varepsilon=\varepsilon_{n},\ x=0$, we have
\begin{equation}\label{t5.3.3}
2\theta\leq\varepsilon_{n}\hat{u}^{\varepsilon_{n}}(x_{n})\leq\varepsilon_{n}\hat{u}^{\varepsilon_{n}}(x_{m})+\eta_{m,n}
\leq\varepsilon_{n}(u^{\varepsilon_{n}}(x_{m})-u^{\varepsilon_{n}}(0))+\eta_{m,n}
\end{equation}
for any $n,m\in\mathbf{N}$.
On the other hand, from Remark \ref{re5.1}, we have
\begin{equation*}
\begin{split}
e^{\tilde{u}(y)-\tilde{u}(x)}
&\leq\frac{1}{C(a(y)+\varepsilon u^{\varepsilon}(y))}\\
&\leq\frac{1}{C(\inf\limits_{x}a(x)+\inf\limits_{x}\{\int_{\mathbf{R}}K(x,y)e^{p(y-x)}dy-a(x)\})}\\
&\leq\frac{1}{C(\inf\limits_{x}\int_{\mathbf{R}}K(x,y)e^{p(y-x)}dy-\sup\limits_{x}a(x)+\inf\limits_{x}a(x))}\\
&\leq\frac{1}{\kappa C}
\end{split}
\end{equation*}
in $(S1)$. Hence $\eta_{m,n}\leq\|a(\cdot+x_{n})-a(\cdot+x_{m})\|_{\infty}+\frac{1}{\kappa C}
\sup\limits_{x\in\mathbf{R}}\int_{B_{r_{0}(x)}}
|\big( K_{p}(x+x_{n},y+x_{n})-K_{p}(x+x_{m},y+x_{m})\big)|dy$ in $(S1)$ and
$\eta_{m,n}\leq\|a(\cdot+x_{n})-a(\cdot+x_{m})\|_{\infty}$ in $(S2)$.
Therefore $\eta_{m,n}\to0$ as $m,n\to\infty$ uniformly with respect to $\varepsilon$ by the choice of $x_{n}$.
One can find $n_{0}$ such that $\eta_{m,n}<\theta$ $\forall m,n\geq n_{0}$. In particular,
$\eta_{n_{0},n}<\theta$ $\forall n\geq n_{0}$. Hence from \eqref{t5.3.3}, we have
$$2\theta\leq\varepsilon_{n}(u^{\varepsilon_{n}}(x_{n_{0}})-u^{\varepsilon_{n}}(0))+\eta_{n_{0},n}
\leq\varepsilon_{n}(u^{\varepsilon_{n}}(x_{n_{0}})-u^{\varepsilon_{n}}(0))+\theta\to\theta$$
as $n\to\infty$ by Corollary \ref{cor5.1}, which is a contradiction!
Thus we complete the proof of claim.

The claim means that for any sequence $\{\varepsilon_{n}\}$ there exists a subsequence still denoted by $\{\varepsilon_{n}\}$ such that
$\varepsilon_{n}u^{\varepsilon_{n}}\rightrightarrows
\lim\limits_{n\to+\infty}\varepsilon_{n}u^{\varepsilon_{n}}(0)$.
Then we still need to show that for any sequence $\varepsilon_{n}$ trending to $0$,
$\varepsilon_{n}u^{\varepsilon_{n}}$ converges to the same constant as $n\to+\infty$.
If not, then exist $\{\varepsilon_{n}\}$ and $\{\varepsilon^{\prime}_{n}\}$ such that
$\varepsilon_{n}u^{\varepsilon_{n}}\rightrightarrows a,\ \text{and}\
 \varepsilon^{\prime}_{n}u^{\varepsilon^{\prime}_{n}}\rightrightarrows b$
as $n\to+\infty$. Without loss of generality, we may assume $a>b$. Then we choose
$\varepsilon\in\{\varepsilon_{n}\},\ \varepsilon^{\prime}\in\{\varepsilon^{\prime}_{n}\}$ such that
$\|\varepsilon u^{\varepsilon}-a\|_{{\infty}}<\frac{a-b}{4},\
 \|\varepsilon^{\prime}u^{\varepsilon^{\prime}}-b\|_{{\infty}}<\frac{a-b}{4}$. Hence
$\varepsilon u^{\varepsilon}(x)-\varepsilon^{\prime}u^{\varepsilon^{\prime}}(y)>\frac{a-b}{2}$
$\forall\ x,y\in\mathbf{R}$. Let
$$\Phi(x,y)=u^{\varepsilon}(x)-u^{\varepsilon^{\prime}}(y)-\alpha|x-y|-\mu(|x|+|y|).$$
Then by the same argument as Theorem \ref{thm5.1},
$\Phi$ reaches its maximum at some point, say $(\xi,\eta)$,
over $\mathbf{R}^{2}$. $(\xi,\eta)$ depends on $\alpha\ \text{and}\ \mu$,
and $|\xi-\eta|\to0$ as $\alpha\to+\infty$ uniformly with respect to $\mu$. One can finally find that
$$0<\frac{a-b}{2}
\leq \int_{\mathbf{R}}K_{p}(\xi,y)e^{u^{\varepsilon}(y)-u^{\varepsilon}(\xi)}dy-a(\xi)
-\int_{\mathbf{R}}K_{p}(\eta,y)e^{u^{\varepsilon^{\prime}}(y)-u^{\varepsilon^{\prime}}(\eta)}dy+a(\eta)\to0$$
as $\mu\to0$ and $\alpha\to+\infty$ (after passing a subsequence), which is a contradiction! Thus the proof is complete.
\end{proof}

Denote $\lambda_{0}:=\lim\limits_{\varepsilon\to0}\varepsilon u^{\varepsilon}$, now we can prove our
main result of this section.

\begin{thm}\label{thm5.4}
Under the assumptions in Theorem \ref{thm5.3} (except that $a$ is locally Lipschitz continuous). If
$K\varphi$ is uniformly continuous for any $\varphi\in L^{\infty}(\mathbf{R})\cap C(\mathbf{R})$,
then we have
$$\underline{\omega}=\overline{\omega}.$$
\end{thm}
\begin{proof}
Step 1: We prove the theorem when $a$ is locally Lipschitz continuous. Let $\phi(x)=\mathrm{e}^{u^{\varepsilon}(x)}$, where $u^{\varepsilon}$ is a solution of \eqref{5.1}.
Obviously, $\phi\in\mathcal{A}$, and
$(L_{p}\phi)(x)=\varepsilon u^{\varepsilon}(x)\phi(x)$.
Moreover, for any $\kappa>0$ small, there exists $\varepsilon_{0}$ such that
$\|\varepsilon u^{\varepsilon}-\lambda_{0}\|_{{\infty}}\leq\kappa$ $\forall\varepsilon\leq\varepsilon_{0}$. Then from the definition and
the monotonicity of $\underline{\lambda}_{1},\overline{\lambda}_{1}$,
one can take $\phi$ as a test function to obtain
$$\lambda_{0}-\kappa\leq\underline{{\lambda}_{1}}(p,-\infty)\leq
\underline{{\lambda}_{1}}(p,R)\leq\overline{{\lambda}_{1}}(p,R)
\leq\overline{{\lambda}_{1}}(p,-\infty)\leq\lambda_{0}+\kappa$$
for $\kappa>0,\ R\in\mathbf{R},\ p\in\mathbf{R}$. Setting $\kappa\to0$, we have
$$\lambda_{0}=\underline{{\lambda}_{1}}(p,-\infty)=\underline{{\lambda}_{1}}(p,n)
=\overline{{\lambda}_{1}}(p,n)=\overline{{\lambda}_{1}}(p,-\infty).$$

Step 2: Now consider general $a(x)$. Let the standard mollifier
$$\rho(x)=
\left\{
   \begin{aligned}
 C\exp(\frac{1}{|x|^{2}-1}),\ & \ \ \ |x|<1,\\
               0,\ & \ \ \ |x|\geq1,\\
   \end{aligned}
   \right.$$
for some constant $C$ such that $\int_{\mathbf{R}}\rho(y)dy=1$. Denote $\rho_{n}(\cdot)=n\rho(n\cdot).$ Then one can easily check that
$a_{n}(x):=\int_{\mathbf{R}}\rho_{n}(x-y)a(y)$ is Lipschitz continuous for any $n\in\mathbf{N}$
since $\sup\limits_{x\in\mathbf{R}}|a^{\prime}_{n}(x)|<\infty$ and there exists a subsequence still
denoted by $n$ such that $a_{n}(\cdot)\to a(\cdot)$ in $C(\mathbf{R})$ since $a$ is almost priodic.
Then, by Step 1,
$\underline{{\lambda}_{1}}(p,-\infty,a_{n})=\overline{{\lambda}_{1}}(p,-\infty,a_{n}).$
Combining this with Propsition \ref{prop2.2}, one can find that:
$$\underline{{\lambda}_{1}}(p,-\infty,a_{n})-\|a_{n}-a\|_{\infty}\leq
\underline{{\lambda}_{1}}(p,-\infty,a)\leq\overline{{\lambda}_{1}}(p,-\infty,a)
\leq\overline{{\lambda}_{1}}(p,-\infty,a_{n})+\|a_{n}-a\|_{\infty}$$
for any $n\in\mathbf{N}$. Hence
$$\lim\limits_{n\to\infty}\underline{{\lambda}_{1}}(p,-\infty,a_{n})=
\underline{{\lambda}_{1}}(p,-\infty,a)=\overline{{\lambda}_{1}}(p,-\infty,a)
=\lim\limits_{n\to\infty}\overline{{\lambda}_{1}}(p,-\infty,a_{n}).$$
\end{proof}

If, furthermore, $K(x,y)=K(y,x)$, then we can prove that the speed in the positive direction
equals to the speed in the negative direction, i.e.,
$\underline{\omega}^{-}=\overline{\omega}^{-}=\underline{\omega}=\overline{\omega}$, where
$\underline{\omega}^{-},\ \overline{\omega}^{-}$ were given in Remark \ref{re2.1}. In fact, we have the following theorem:
\begin{thm}\label{thm5.5}
Under the assumptions in Theorem \ref{thm5.3}. If $K(x,y)=K(y,x)$, then
$$
\underline{{\lambda}_{1}}(p,-\infty)=\overline{{\lambda}_{1}}(p,-\infty)=
\underline{{\lambda}^{-}_{1}}(p,-\infty)=\overline{{\lambda}^{-}_{1}}(p,-\infty).
$$
\end{thm}
\begin{proof}
First by Theorem \ref{thm5.4}, there exist
$u^{\varepsilon}, v^{\varepsilon}\in C{(\mathbf{R})}\cap L^{\infty}(\mathbf{R})$ such that
$$\left\{
   \begin{aligned}
    \int_{\mathbf{R}}K(x,y)e^{p(y-x)}\phi(y)dy-a(x)\phi(x)
    =\varepsilon u^{\varepsilon}(x)\phi(x),\ x\in\mathbf{R},\\
    \int_{\mathbf{R}}K^{-}(x,y)e^{p(y-x)}\psi(y)dy-a^{-}(x)\psi(x)
    =\varepsilon u^{\varepsilon}(x)\psi(x),\ x\in\mathbf{R},
   \end{aligned}
   \right.$$
where $\phi=e^{u^{\varepsilon}},\ \psi=e^{v^{\varepsilon}}\in\mathcal{A}$ and $a^{-}(x)=a(-x)$.
Moreover,
$$\underline{{\lambda}_{1}}(p,-\infty)=\overline{{\lambda}_{1}}(p,-\infty)=\lim\limits_{\varepsilon\to0}
\varepsilon u^{\varepsilon}$$
$$\underline{{\lambda}^{-}_{1}}(p,-\infty)=\overline{{\lambda}^{-}_{1}}(p,-\infty)=\lim\limits_{\varepsilon\to0}
\varepsilon v^{\varepsilon}.$$
We denote $\lambda_{0}:=\lim\limits_{\varepsilon\to0}\varepsilon u^{\varepsilon}$ and
$\lambda^{-}_{0}:=\lim\limits_{\varepsilon\to0}\varepsilon v^{\varepsilon}$.
Now it is sufficient to show
$\lambda_{0}=\lambda^{-}_{0}$. If not, we may, without loss of generality, assume by contradiction that
$\lambda_{0}<\lambda^{-}_{0}$, then there exists $\varepsilon_{0}$ such that
$\varepsilon_{0} u^{\varepsilon_{0}}<\lambda_{0}+\frac{\lambda^{-}_{0}-\lambda_{0}}{4}$ and
$\varepsilon_{0} v^{\varepsilon_{0}}>\lambda^{-}_{0}-\frac{\lambda^{-}_{0}-\lambda_{0}}{4}$. Denote $d_{0}:=\frac{\lambda^{-}_{0}-\lambda_{0}}{4}$. Hence
$$\left\{
   \begin{aligned}
    \int_{\mathbf{R}}K(x,y)e^{p(y-x)}\phi(y)dy-a(x)\phi(x)
    =\varepsilon_{0} u^{\varepsilon_{0}}(x)\phi(x)
    \leq(\lambda_{0}+d_{0})\phi(x),\ x\in\mathbf{R},\\
    \int_{\mathbf{R}}K^{-}(x,y)e^{p(y-x)}\psi(y)dy-a^{-}(x)\psi(x)
    =\varepsilon_{0} u^{\varepsilon_{0}}(x)\psi(x)
    \geq(\lambda^{-}_{0}-d_{0})\psi(x),\ x\in\mathbf{R}.
   \end{aligned}
   \right.$$
Therefore,
\begin{equation}\label{t5.5.1}
\int_{-r}^{r}\psi(-x)\int_{\mathbf{R}}K(x,y)e^{p(y-x)}\phi(y)dydx-\int_{-r}^{r}a(x)\phi(x)\psi(-x)dx
    \leq(\lambda_{0}+d_{0})\int_{-r}^{r}\phi(x)\psi(-x)dx,
\end{equation}
\begin{equation}\label{t5.5.2}
\int_{-r}^{r}\phi(x)\int_{\mathbf{R}}K(x,y)e^{-p(y-x)}\psi(-y)dydx-\int_{-r}^{r}a(x)\phi(x)\psi(-x)dx
    \geq(\lambda^{-}_{0}-d_{0})\int_{-r}^{r}\phi(x)\psi(-x)dx
\end{equation}
for any $r\in\mathbf{R}$. \eqref{t5.5.2} follows from $K(x,y)=K(y,x)$.
By \eqref{t5.5.2}-\eqref{t5.5.1}, we have
\begin{equation}\label{t5.5.3}
\begin{split}
2d_{0}\int_{-r}^{r}\phi(x)\psi(-x)dx
&\leq\int_{-r}^{r}\phi(x)\int_{\mathbf{R}}K(x,y)e^{-p(y-x)}\psi(-y)dydx\\
&\quad-\int_{-r}^{r}\psi(-x)\int_{\mathbf{R}}K(x,y)e^{p(y-x)}\phi(y)dydx\\
&=\int_{-r}^{r}\phi(y)\int_{\mathbf{R}}K_{p}(x,y)\psi(-x)dxdy\\
&\quad-\int_{-r}^{r}\psi(-x)\int_{\mathbf{R}}K_{p}(x,y)\phi(y)dydx\\
&\leq\int_{-r}^{r}\phi(y)\int_{B^{C}_{r}(0)}K_{p}(x,y)\psi(-x)dxdy\\
&=\int_{-r}^{r}\phi(y)\int_{r}^{+\infty}K_{p}(x,y)\psi(-x)dxdy\\
&\quad+\int_{-r}^{r}\phi(y)\int_{-\infty}^{-r}K_{p}(x,y)\psi(-x)dxdy.\\
\end{split}
\end{equation}
Denote the first term of the right hand side of \eqref{t5.5.3} by $A_{1}$ and the second term by $A_{2}$.
Now there exists $r_{0}$ so that
$\int_{B^{c}_{r_{0}}(x)}K(x,y)e^{p(y-x)}dy\leq d_{0}\frac{\inf\limits_{x\in\mathbf{R}}\psi(x)}{2\sup\limits_{x\in\mathbf{R}}\psi(x)}.$
Then
\begin{equation*}
\begin{split}
A_{1}
&=\int_{-r}^{r}\phi(y)\int_{r}^{r+r_{0}}K_{p}(x,y)\psi(-x)dxdy
+\int_{-r}^{r}\phi(y)\int_{r+r_{0}}^{+\infty}K_{p}(x,y)\psi(-x)dxdy\\
&\leq r_{0}\sup\limits_{x\in\mathbf{R}}\psi(x)\sup\limits_{x\in\mathbf{R}}\phi(x)
 \int_{-r}^{r}K_{p}(x,y)dy+\frac{d_{0}}{2}\inf\limits_{x\in\mathbf{R}}\psi(x)\int_{-r}^{r}\phi(y)dy.\\
\end{split}
\end{equation*}
Similarly,
$A_{2}\leq r_{0}\sup\limits_{x\in\mathbf{R}}\psi(x)\sup\limits_{x\in\mathbf{R}}\phi(x)
 \int_{-r}^{r}K_{p}(x,y)dy+\frac{d_{0}}{2}\inf\limits_{x\in\mathbf{R}}\psi(x)\int_{-r}^{r}\phi(y)dy.$
Combining this with \eqref{t5.5.3}, we have
\begin{equation*}
\begin{split}
2d_{0}\inf\limits_{x\in\mathbf{R}}\psi(x)\int_{-r}^{r}\phi(y)dy
&\leq2d_{0}\int_{-r}^{r}\phi(x)\psi(-x)dx\leq A_{1}+A_{2}\\
&\leq2r_{0}\sup\limits_{x\in\mathbf{R}}\psi(x)\sup\limits_{x\in\mathbf{R}}\phi(x)
 \int_{-r}^{r}K_{p}(x,y)dy+{d_{0}}\inf\limits_{x\in\mathbf{R}}\psi(x)\int_{-r}^{r}\phi(y)dy.\\
\end{split}
\end{equation*}
Hence \begin{equation}\label{t5.5.4}
d_{0}\inf\limits_{x\in\mathbf{R}}\psi(x)\int_{-r}^{r}\phi(y)dy
\leq2r_{0}\sup\limits_{x\in\mathbf{R}}\psi(x)\sup\limits_{x\in\mathbf{R}}\phi(x)\int_{-r}^{r}K_{p}(x,y)dy,
\end{equation}
which is a contradiction since the right hand side of \eqref{t5.5.4} $\to+\infty$ as $r\to+\infty$
while the left hand side of \eqref{t5.5.4} is bounded!
\end{proof}

\subsection{Periodic coefficients}
In this subsection, we always assume that $a(\cdot)\in\ C(\mathbf{R})$ is periodic function with period $L$.
We consider a special class of kernel $K(x,y)=\sum\limits_{n=1}^{N}a_{n}\delta_{q_{n}}(x-y)$,
where $N\in\mathbf{N}\cup\{+\infty\}$, $a_{n}$ is a positive sequence, $q_{n}\in\mathbf{R}$
and $\delta$ is the Dirac's delta function on $\mathbf{R}$.
Assume that $K$ satisfies the following hypotheses:\\
(H1) $\sum\limits_{n=1}^{N}a_{n}e^{-pq_{n}}<+\infty\ \forall p\in\mathbf{R}.$\\
(H2) $\inf_{p\in\mathbf{R}}\sum\limits_{n=1}^{N}a_{n}e^{-pq_{n}}-\sum\limits_{n=1}^{N}a_{n}+
\lim\limits_{R\to+\infty}\inf\limits_{x\in I_{R}}f^{\prime}_{s}(x,0)>0.$\\
(H3) The assumptions of Example {eg4.2} hold.\\
(H4) Either (1) $\exists q_{n_{3}}$ such that $\frac{q_{n_{3}}}{L}\notin\mathbf{Q} (i=1,2)$ or

\ \ \ \ \ \ \ \ \ \ \ \ \ (2) $\exists (a,b)$ such that $(a,b)\subset\overline{\{q_{n}\}_{n}}$.

The condition (H1) and (H2) are corresponding to (K1) and (K4) respectively.
\begin{thm}\label{thm5.6}
Under the above assumptions. We have
$$\underline{\omega}=\overline{\omega}.$$
\end{thm}
\begin{proof}
The proof was divided into three steps:

Step 1: The Perron's method gives us a unique periodic solution $u^{\varepsilon}\in C(\mathbf{R})$ of
$$\big(\varepsilon u^{\varepsilon}(x)+a(x)\big)\phi(x)=
\int_{\mathbf{R}}K_{p}(x,y)\phi(y)dy=\sum\limits_{n=1}^{N}a_{n}e^{-pq_{n}}\phi(x-q_{n}),$$
where $\phi(x)=e^{u^{\varepsilon}(x)}.$ Moreover, the period of $u^{\varepsilon}$ is $L$.
It is sufficient to show that
$\varepsilon u^{\varepsilon}(x)$ converges to some constant as $\varepsilon\to0$
uniformly with respect to $x\in\mathbf{R}$.

Step 2: For any $q_{k}\in\{q_{n}\}, m\in\mathbf{Z}$, we have
\begin{equation}\label{t5.6.1}
\big(\varepsilon u^{\varepsilon}(mq_{k})+a(mq_{k})\big)\phi(mq_{k})
=\sum\limits_{n=1}^{N}a_{n}e^{-pq_{n}}\phi(mq_{k}-q_{n}).
\end{equation}
Hence, from \eqref{t5.6.1}, we have
$$\big(\varepsilon u^{\varepsilon}(mq_{k})+a(mq_{k})\big)\phi(mq_{k})
\geq a_{k}e^{-pq_{k}}\phi((m-1)q_{k}),$$
which yields
$$e^{u^{\varepsilon}(mq_{k})-u^{\varepsilon}((m-1)q_{k})}
\geq\frac{a_{k}e^{-pq_{k}}}{\varepsilon u^{\varepsilon}(mq_{k})+a(mq_{k})}
\geq\frac{a_{k}e^{-pq_{k}}}{\overline{c}+a(mq_{k})},$$
the latter inequality following from \eqref{subsuper}. Therefore,
$$\liminf\limits_{\varepsilon\to0^{+}}\varepsilon(u^{\varepsilon}(mq_{k})-u^{\varepsilon}((m-1)q_{k}))\geq
\liminf\limits_{\varepsilon\to0^{+}}\{\varepsilon\ln(a_{k}e^{-pq_{k}})-\varepsilon\ln(\overline{c}+a(mq_{k}))\}
=0.$$
Now it is easy to find that
$$\liminf\limits_{\varepsilon\to0^{+}}\varepsilon(u^{\varepsilon}(mq_{k})-u^{\varepsilon}(0))\geq0,\
\forall m=1,2,\cdots,$$
and
$$\liminf\limits_{\varepsilon\to0^{+}}\varepsilon(u^{\varepsilon}(0)-u^{\varepsilon}(mq_{k}))\geq0,\
\forall m=-1,-2,\cdots,$$
i.e., \begin{equation}\label{t5.6.2}
\limsup\limits_{\varepsilon\to0^{+}}\varepsilon(u^{\varepsilon}(-mq_{k})-u^{\varepsilon}(0))\leq0,\
\forall m=1,2,\cdots.
\end{equation}
Form the assumption (H3), we have
$[0,L]\subset\overline{A_{\pm}},$ where $A_{\pm}:=\bigcup\limits_{m=1}^{\infty}\pm m\{q_{n}\}(\mod L)$.

Step 3: First, we show that $\varepsilon(u^{\varepsilon}(x)-u^{\varepsilon}(0))\to0$ as
$\varepsilon\to0$ uniformly with respect to $x$.
Denote $\varepsilon\hat{u}^{\varepsilon}(x):=\varepsilon(u^{\varepsilon}(x)-u^{\varepsilon}(0)).$
If not, there exist $\varepsilon_{n}\to0,\ x_{n}\in[0,L],\ \theta>0$ such that
$|\varepsilon_{n}\hat{u}^{\varepsilon_{n}}(x_{n})|\geq2\theta$. We will prove the case where
$\varepsilon_{n}\hat{u}^{\varepsilon_{n}}(x_{n})>2\theta$, and one can prove the case where
$\varepsilon_{n}\hat{u}^{\varepsilon_{n}}(x_{n})<-2\theta$ similarly.
Since $u^{\varepsilon}(x)\in C(\mathbf{R})$ and $[0,L]\subset\overline{A_{\pm}},$
one can choose $x_{n}\in A_{-}$. As we did in the proof of
Theorem \ref{thm5.3}, we have
$$2\theta\leq\varepsilon_{n}\hat{u}^{\varepsilon_{n}}(x_{n})\leq\varepsilon_{n}\hat{u}^{\varepsilon_{n}}(x_{m})+\eta_{m,n}
\leq\varepsilon_{n}(u^{\varepsilon_{n}}(x_{m})-u^{\varepsilon_{n}}(0))+\eta_{m,n},$$
where $\eta_{m,n}=\|a(\cdot+x_{n})-a(\cdot+x_{m})\|_{\infty}$ with
$\eta_{m,n}\to0$ as $m,n\to\infty$. Hence there exists $n_{0}$ such that
$$2\theta\leq\varepsilon_{n}(u^{\varepsilon_{n}}(x_{n_{0}})-u^{\varepsilon_{n}}(0))+\eta_{n_{0},n}
\leq\varepsilon_{n}(u^{\varepsilon_{n}}(x_{n_{0}})-u^{\varepsilon_{n}}(0))+\theta,\ \ \forall n\geq n_{0}.$$
Then
$\limsup\limits_{\varepsilon\to0^{+}}\varepsilon(u^{\varepsilon}(x_{n_{0}})-u^{\varepsilon}(0))\geq\theta$,
which contradicts \eqref{t5.6.2}!

Finally, by a similar argument to Theorem \ref{thm5.3}, one can prove that
$\varepsilon u^{\varepsilon}(x)$ converges to some constant as $\varepsilon\to0$
uniformly with respect to $x\in\mathbf{R}$.
\end{proof}

\begin{thm}
Under the assumptions in Theorem \ref{thm5.6}. If $K(x,y)=K(y,x)$, then
$$
\underline{{\lambda}_{1}}(p,-\infty)=\overline{{\lambda}_{1}}(p,-\infty)=
\underline{{\lambda}^{-}_{1}}(p,-\infty)=\overline{{\lambda}^{-}_{1}}(p,-\infty).
$$
\end{thm}
\begin{proof}
The proof is similar to Theorem \ref{thm5.5}.
\end{proof}

\end{document}